\documentclass[a4paper,12pt]{amsart}
\usepackage[a4paper,margin=2.4cm]{geometry}
\usepackage{color}
\usepackage{amssymb}
\usepackage{amsmath}
\AtEndDocument{\vfill\eject\batchmode}
\newtheorem{thm}{Theorem}
\newtheorem{lemma}{Lemma}
\newtheorem{prop}{Proposition}
\newtheorem{cor}{Corollary}
\theoremstyle{definition}
\newtheorem{defn}{Definition}
\newtheorem{problem}{Problem}
\theoremstyle{remark}
\newtheorem{rem}{Remark}
\newtheorem{rems}[rem]{Remarks}
\newtheorem{ex}{Example}

\newcounter{numl}
\newcommand{\labelnuml}{\textup{(\roman{numl})}}
\newenvironment{numlist}{\begin{list}{\labelnuml}%
{\usecounter{numl}\setlength{\leftmargin}{0pt}%
\setlength{\itemindent}{2\parindent}%
\setlength{\itemsep}{\smallskipamount}\def
\makelabel ##1{\hss \llap {\upshape ##1}}}}{\end{list}}
\newenvironment{bulletlist}{\begin{list}{\labelitemi}%
{\setlength{\leftmargin}{\parindent}\def
\makelabel ##1{\hss \llap {\upshape ##1}}}}{\end{list}}
\DeclareSymbolFont{script}{U}{eus}{m}{n}
\DeclareSymbolFontAlphabet{\mathscr}{script}
\DeclareMathSymbol{\Wedge}{0}{script}{"5E}
\DeclareMathAlphabet{\mathrmsl}{OT1}{cmr}{m}{sl}
\renewcommand{\geq}{\geqslant}
\renewcommand{\leq}{\leqslant}
\newcommand{\R}{{\mathbb R}}
\newcommand{\C}{{\mathbb C}}
\newcommand{\Z}{{\mathbb Z}}

\newcommand{\Q}{{\mathbb Q}}
\newcommand{\T}{{\mathbb T}}
\newcommand{\cD}{{\mathcal D}}

\newcommand{\cK}{{\mathcal K}}
\newcommand{\cL}{{\mathcal L}}
\newcommand{\cO}{{\mathcal O}}
\newcommand{\cS}{{\mathcal S}}
\newcommand{\eps}{\varepsilon}

\newcommand{\Lam}{{\mathrmsl\Lambda}}
\newcommand{\su}{{\mathfrak{su}}}
\newcommand{\crJ}{{\mathfrak{cr}}}
\newcommand{\con}{\mathfrak{con}}
\newcommand{\tor}{{\mathfrak t}}
\newcommand{\torh}{{\mathfrak h}}
\newcommand{\sub}{\subseteq}

\newcommand{\vol}{\mathit{vol}}
\newcommand{\restr}[1]{|_{#1}^{\vphantom x}}
\newcommand{\into}{\hookrightarrow}

\newcommand{\spn}[1]{\mathop{\mathrmsl{span}}(#1)}
\newcommand{\trace}{\mathop{\mathrm{tr}}\nolimits}

\newcommand{\ip}[1]{\langle #1 \rangle}
\newcommand{\Proj}{\mathrmsl P}
\newcommand{\End}{\mathrmsl{End}}

\newcommand{\Ham}{\mathrmsl{Ham}}
\newcommand{\Con}{\mathrmsl{Con}}
\newcommand{\Aut}{\mathrmsl{Aut}}
\newcommand{\Aff}{\mathrmsl{Aff}}
\renewcommand{\d}{{\mathrmsl d}}

\newcommand{\grad}{\mathop{\mathrmsl{grad}}\nolimits}
\newcommand{\Hess}{\mathrmsl{Hess}}
\newcommand{\Scal}{\mathit{Scal}}
\newcommand{\Bm}{B}
\newcommand{\Sm}{N}
\newcommand{\Dm}{U}
\newcommand{\Sph}{{\mathbb S}}
\newcommand{\Ds}{{\mathscr D}}
\newcommand{\Lv}{L}
\newcommand{\cf}{\eta}
\newcommand{\sas}{\chi}
\newcommand{\cae}{\xi}
\newcommand{\mm}{\mu}
\newcommand{\mc}{z}
\newcommand{\wv}{\mathrmsl{w}}
\newcommand{\wt}{\nu}
\newcommand{\pt}{p}
\newcommand{\Sc}{\Xi}
\newcommand{\J}{{\mathbf J}}
\newcommand{\Om}{{\mathbf\Omega}}
\newcommand{\As}{{\mathcal A}}
\newcommand{\op}{^\circ}
\newcommand{\bx}{\boldsymbol{x}}
\newcommand{\Afn}{A}
\newcommand{\ah}{a}
\newcommand{\bh}{b}
\newcommand{\bht}{c_{\ah,\bh}}
\newcommand{\ca}{p}
\newcommand{\av}{\kappa}
\newcommand{\genus}{{\boldsymbol g}}
\newcommand{\p}{k}
\newcommand{\q}{n}

\begin{document}

\title[CR geometry of weighted extremal metrics]{The CR geometry\\
of weighted extremal K\"ahler and Sasaki metrics}

\author[V. Apostolov]{Vestislav Apostolov}
\address{Vestislav Apostolov \\ D{\'e}partement de Math{\'e}matiques\\
UQAM\\ C.P. 8888 \\ Succursale Centre-ville \\ Montr{\'e}al (Qu{\'e}bec) \\
H3C 3P8 \\ Canada}
\email{apostolov.vestislav@uqam.ca}
\author[D.M.J. Calderbank]{David M. J. Calderbank}
\thanks{VA was supported in part by an NSERC Discovery Grant. The authors are grateful to the Institute of Mathematics and Informatics of the Bulgarian Academy of Science where some of this work was conducted.}
\address{David M. J. Calderbank \\ Department of Mathematical Sciences\\
University of Bath\\ Bath BA2 7AY\\ UK}
\email{D.M.J.Calderbank@bath.ac.uk}

\begin{abstract} We establish an equivalence between conformally
Einstein--Maxwell K\"ahler $4$-manifolds recently studied in
\cite{ambitoric1,AM,FO0,KTF,LeB0,LeB1,LeB2} and extremal K\"ahler
$4$-manifolds in the sense of Calabi~\cite{calabi} with nowhere vanishing
scalar curvature.  The corresponding pairs of K\"ahler metrics arise as
transversal K\"ahler structures of Sasaki metrics compatible with the same CR
structure and having commuting Sasaki--Reeb vector fields. This correspondence
extends to higher dimensions using the notion of a weighted extremal K\"ahler
metric~\cite{ACGL,AMT,lahdili1,lahdili,lahdili2}, illuminating and uniting
several explicit constructions in K\"ahler and Sasaki geometry. It also leads
to new existence and non-existence results for extremal Sasaki metrics,
suggesting a link between the notions of relative weighted K-stability of a
polarized variety introduced in \cite{AMT,lahdili2}, and relative K-stability
of the K\"ahler cone corresponding to a Sasaki polarization, studied in
\cite{BV,CSz}.
\end{abstract}

\maketitle

\section*{Introduction}

The famous Calabi problem~\cite{calabi}, which seeks the existence of
canonical K\"ahler metrics, is a central and very active topic of current
research in K\"ahler geometry.  As a candidate for a canonical metric on a
complex manifold $(M,J)$, Calabi proposed a notion of \emph{extremal K\"ahler
  metric} $g$, meaning that its scalar curvature $\Scal(g)$ is a \emph{Killing
  potential}, i.e., the vector field $J\grad_g \Scal(g)$ is a Killing vector
field for $g$. Examples include constant scalar curvature (CSC) K\"ahler
metrics on $(M,J)$, and hence also K\"ahler--Einstein metrics.

More recently, in real dimension $4$, another natural generalization of CSC
K\"ahler metrics has been studied~\cite{ambitoric1,LeB0,LeB1,LeB2}: K\"ahler
metrics $g$ admitting a positive Killing potential $f$ for which the scalar
curvature of the conformal metric $\tilde g = (1/f^2) g$ is a constant $c$,
i.e.,
\begin{equation}\label{EM}
\Scal(\tilde g) = f^2 \Scal(g) - 6 f\Delta_g f - 12 |\d f|^2_g = c, 
\end{equation}
where $\Delta_g$ is the riemannian laplacian and $|\cdot|_g$ the norm defined
by $g$.  The metric $\tilde g$ satisfies a riemannian analogue of the
Einstein--Maxwell equations with cosmological constant in generally
relativity~\cite{PD,DKM}, and thus we say $g$ is a \emph{conformally
  Einstein--Maxwell K\"ahler metric}.  Many explicit examples of such metrics
have been exhibited~\cite{ambitoric1,ambitoric2, AM,FO0,KTF,LeB1,LeB2}, and
they have a striking resemblance to similar explicit examples of extremal
K\"ahler metrics, see e.g.~\cite[Prop.~3]{AM}. Elucidating the connection
suggested by these examples was the main motivation for this article, and our
main result implies in particular an equivalence between the classes of
conformally Einstein--Maxwell K\"ahler $4$-manifolds and extremal K\"ahler
$4$-manifolds of nowhere zero scalar curvature. Our approach was suggested in
part by~\cite[App.~C]{ambitoric2}, which implies that both kinds of metric can
arise as quotients of a common strictly pseudo-convex CR $5$-manifold $(\Sm,
\Ds, J)$ of Sasaki type. It also generalizes to complex manifolds $(M,J)$ of
real dimension $2m$, using the notion of a \emph{weighted extremal
  metric}~\cite{ACGL,AM,lahdili1}, as we now explain.

Let $(g,\omega)$ be a K\"ahler metric on $(M,J)$, $f$ a function on $M$, and
$\wt\in\R$ a real number (which we call the \emph{weight}). Then the
\emph{$(f,\wt)$ scalar curvature} of $g$ is defined to be
\begin{equation}\label{(f,wt)-scalar-curvature}
\Scal_{f,\wt}(g) := f^2\Scal(g) - 2(\wt-1) f\Delta_g f - \wt(\wt-1)|\d f|^2_g,
\end{equation}

\begin{defn}\label{d:(f,wt)-extremal} Let $(g,\omega)$ be a K\"ahler metric
on $(M,J)$ let $f$ be a Killing potential for $g$ and let $\wt\in\R$.  We say
that $g$ is \emph{$(f,\wt)$-extremal} if its $(f,\wt)$ scalar curvature,
$\Scal_{f,\wt}(g)$, given by \eqref{(f,wt)-scalar-curvature}, is also a
Killing potential.
\end{defn}

When $M$ is compact and $f>0$ on $M$, the $(f,\wt)$ scalar
curvature~\eqref{(f,wt)-scalar-curvature} is the momentum map associated to a
formal GIT problem on the space $\cK_\omega(M)^\T$ of $\T$-invariant
$\omega$-compatible K\"ahler metrics for any torus $\T$ in the group
$\Ham(M,\omega)$ of hamiltonian symplectomorphisms of $(M,\omega)$ which
contains the flow of $K$~\cite{ACGL,AM,lahdili1}. This is similar to the framework
found by Donaldson~\cite{donaldson} and Fujiki~\cite{fujiki} for Calabi
extremal K\"ahler metrics. Indeed, the latter can be recovered from the
weighted generalization by setting $f\equiv 1$.

For $m=2$ and $\wt=4$, \eqref{(f,wt)-scalar-curvature} reduces to \eqref{EM},
so that Definition~\ref{d:(f,wt)-extremal} includes the conformally
Einstein--Maxwell K\"ahler metrics already discussed.  This case was extended
to the weight $\wt=2m$ (for any $m$) in~\cite{AM}, where it was noted
that~\eqref{(f,wt)-scalar-curvature} then computes the scalar curvature of the
hermitian metric $(1/f^2) g$. Thus examples of $(f, 2m)$-extremal metrics
include the conformally Einstein K\"ahler metrics studied in
\cite{DM0,DM}. The weight most relevant here is instead $\wt=m+2$, which first
appeared in \cite{ACGL}, where it was discovered that certain quotients, of an
$m$-fold product $\Sph^3\times \cdots \times \Sph^3$ of CR $3$-spheres by an
$m$-torus, are $(f, m+2)$-extremal for a suitable $f$.  However, an intrinsic
geometric interpretation of $(f,m+2)$-extremality with $m>2$ has so far been
lacking. Our main result rectifies this by providing an interpretation in CR
geometry, whose basic notions we now recall (see also
Section~\ref{s:background}).

Let $(\Sm,\Ds)$ be a contact $(2m+1)$-manifold and denote by $\Lv_\Ds\colon
\Ds\times\Ds\to T\Sm/\Ds$ the Levi form of $\Ds$, defined, via local sections
$X,Y\in C^\infty_\Sm(\Ds)$, by the tensorial expression $\Lv_\Ds(X,Y) = -
\cf_\Ds([X,Y])$, where $\cf_\Ds\colon T\Sm\to T\Sm/\Ds$ is the quotient map.
A \emph{contact vector field} is a vector field $X$ such that $\cL_X
(C^\infty_\Sm(\Ds))\sub C^\infty_\Sm(\Ds)$. We make fundamental use of the
following basic fact in the theory of contact manifolds (see
e.g.~\cite{BG-book}).

\begin{lemma}\label{l:contact} The map $X\mapsto\cf_\Ds(X)$ from contact
vector fields to sections of $T\Sm/\Ds$ is a linear isomorphism, whose inverse
$\cae\mapsto X_\cae$ is a first order linear differential operator.
\end{lemma}

\noindent There is thus a \emph{contact Lie algebra} $\con(\Sm,\Ds)$ of
sections $\cae$ of $T\Sm/\Ds$ under the \emph{Jacobi bracket}
\[
[\cae,\sas]:=\cf_\Ds([X_\cae, X_\sas])=\cL_{X_\cae}
\sas=-\cL_{X_{\sas}}\cae.
\]
Now suppose $J\in\End(\Ds)$ is a CR structure on $(N,\Ds)$; then we obtain a
second order linear differential operator $\cae\mapsto \cL_{X_\cae} J$ on
$\con(\Sm,\Ds)$. Its kernel
\[
\crJ(\Sm,\Ds,J):=\{\cae\in \con(\Sm,\Ds): \cL_{X_\cae} J =0\}
\]
is a Lie subalgebra of $\con(\Sm,\Ds)$, whose elements $\cae$ correspond to CR
vector fields $X_\cae$ on $\Sm$. If moreover $(\Ds,J)$ is strictly
pseudo-convex then $TN/\Ds$ has an orientation whose positive sections $\sas$
are those for which $\sas^{-1} \Lv_\Ds(\cdot,J\cdot)$ is positive
definite. Note that $\sas^{-1} \Lv_\Ds=\d\cf_\sas\restr\Ds$ where
$\cf_\sas:=\sas^{-1}\cf_\Ds$ is the contact form defined by $\sas$. We let
$\con_+(\Sm,\Ds)\sub\con(\Sm,\Ds)$ be the open cone of positive sections
$\sas$ of $T\Sm/\Ds$. We then have the following fundamental definitions
(see~e.g.~\cite{BGS}).

\begin{defn}\label{d:Sasaki-cone} Let $(\Sm, \Ds, J)$ be a strictly
pseudo-convex CR manifold. Then the \emph{Sasaki cone} of $(\Sm, \Ds, J)$ is
$\crJ_+(\Sm,\Ds,J):=\crJ(\Sm,\Ds,J)\cap\con_+(\Sm,\Ds)$. If
$\crJ_+(\Sm,\Ds,J)$ is nonempty then $(\Sm,\Ds,J)$ is said to be of
\emph{Sasaki type}, an element $\sas\in\crJ_+(\Sm,\Ds,J)$ is called a
\emph{Sasaki structure} on $(\Sm,\Ds,J)$, with \emph{Sasaki--Reeb vector
  field} $X_\sas$, and $(\Sm,\Ds,J,\sas)$ is called a \emph{Sasaki
  manifold}. We say $\sas$ is \emph{quasi-regular} if the flow of $X_\sas$
generates an $\Sph^1$ action on $\Sm$, and moreover \emph{regular} if this
action is free.
\end{defn}

The following well-known construction provides a standard way (see
e.g.~\cite{BG-book}) to extend geometric notions on K\"ahler manifolds to
Sasaki manifolds.

\begin{ex}\label{e:regular} Let $(M, J, g, \omega)$ be a K\"ahler manifold
such that $[\omega/2\pi]$ is an integral de Rham class.  Then there is a
principal $\Sph^1$-bundle $\pi\colon \Sm\to M$ with a connection $1$-form
$\cf$ satisfying $\d\cf = \pi^* \omega$. Thus $(\Sm,\Ds,J,\sas)$ is a Sasaki
manifold, where $\Ds=\ker\cf\leq T\Sm$, $J$ is the pullback of the complex
structure on $TM$ to $\Ds\cong\pi^*TM$ and $\sas$ is the image in $T\Sm/\Ds$
of the generator $X_\sas$ of the $\Sph^1$ action (with $\cf(X_\sas)=1$, so
$\cf=\cf_\sas$).
\end{ex}

Conversely, if $\sas\in\crJ_+(\Sm,\Ds,J)$ is (quasi-)regular, then $\Sm$ is a
principal $\Sph^1$-bundle (or orbibundle) $\pi\colon \Sm\to M$ over a K\"ahler
manifold (or orbifold) $M$. Irrespective of regularity, this correspondence
between K\"ahler geometry and Sasaki geometry holds locally: any point of a
Sasaki manifold $(\Sm,\Ds,J,\sas)$ has a neighbourhood in which the leaf space
$M$ of the flow of $X_\sas$ is a manifold and has a K\"ahler structure
$(g,J,\omega)$ induced, using the identification $\Ds \cong\pi^*TM$, by the
\emph{transversal K\"ahler structure} $(g_\sas, J, \omega_\sas)$ on $\Ds$,
where $\omega_\sas:=\d\cf_\sas\restr\Ds$ and $g_\sas:= \omega_\sas(\cdot,
J\cdot)$. Indeed $g_\sas$, $J$, and $\omega_\sas$ are all $X_\sas$-invariant,
so they all descend to $M$, and we refer to $(M,g,J,\omega)$ as a
\emph{Sasaki--Reeb quotient} of $(\Sm,\Ds,J,\sas)$.

For $\sas\in \crJ(\Sm, \Ds, J)$, we set
\begin{align*}
\con^\sas&:=\{\cae\in \con(\Sm,\Ds)\,|\, [\sas,\cae]=0\},\\
\crJ^\sas&:=\con^\sas\cap\crJ(\Sm, \Ds, J)\qquad\text{and}\qquad
\crJ_+^\sas:=\con^\sas\cap\crJ_+(\Sm, \Ds, J).
\end{align*}
If in addition $\sas\in \crJ_+(\Sm, \Ds, J)$, then
\[
C^\infty_\Sm(\R)^\sas:=\{f\in C^\infty_\Sm(\R): \d f(X_\sas)=0\}
\]
is a Lie algebra under the \emph{transversal Poisson bracket}
$\{f_1,f_2\}:=-\omega_\sas^{-1} (\d f_1\restr\Ds,\d f_2\restr\Ds)$, and we
have the following elementary but central lemma.

\begin{lemma}\label{l:potential} The map $f\mapsto \cae=f\sas$ is a Lie algebra
isomorphism from $C^\infty_\Sm(\R)^\sas$ to $\con^\sas$, and
$\cae\in\crJ^\sas$ if and only if $f$ is a transversal Killing potential for
$(g_\sas, \omega_\sas)$, i.e., $-\omega_\sas^{-1} (\d f)$ is a transversal
Killing vector field for $g_\sas$.
\end{lemma}

Thus we obtain elements of $\crJ^\sas$ as pullbacks of Killing potentials from
(local) Sasaki--Reeb quotients of $\Sm$ by $\sas$. The Levi-Civita connection
on Sasaki--Reeb quotients pulls back to a connection $\nabla^\sas$ on $\Ds$
preserving $(g_\sas,J,\omega_\sas)$, which turns out to be (see
e.g.~\cite[\S4]{david:weyltanaka}) the so-called \emph{Tanaka--Webster
  connection}~\cite{Webster} of $(\Ds, J,\sas)$.  Thus the scalar curvature of
Sasaki--Reeb quotients pulls back to the \emph{Tanaka--Webster scalar
  curvature} $\Scal(g_\sas)$ of $\nabla^\sas$, and hence $(\Sm,\Ds,J,\sas)$ is
CSC, i.e., $\Scal(g_\sas)$ is constant, if and only if its Sasaki--Reeb
quotients are. We may define (weighted) extremal Sasaki structures similarly.

\begin{defn}\label{d:Sasaki-(xi,wt)-extremal} Let $(\Sm, \Ds, J, \sas)$ be a
Sasaki manifold and $\cae=f\sas \in \crJ^\sas$.  The \emph{$(\cae,\wt)$ scalar
  curvature $\Scal_{\cae,\wt}(g_\sas)$ of $\sas$} is the function induced on
$\Sm$ by the $(f,\wt)$ scalar curvature~\eqref{(f,wt)-scalar-curvature} on
Sasaki--Reeb quotients.  We say that $\sas$ is \emph{$(\cae,\wt)$-extremal} if
$\Scal_{\cae,\wt}(g_\sas)\sas\in \crJ(\Sm,\Ds,J)$. For constant $f$, this
reduces to extremality of $\sas$ in the sense of \cite{BGS}.
\end{defn}
\begin{ex}\label{e:regular-(xi,wt)-extremal} Lemma~\ref{l:potential}
shows that any quasi-regular Sasaki manifold over an $(f,\wt)$-extremal
orbifold $(M, J, g, \omega)$ is $(\cae,\wt)$-extremal, with $\cae=f\sas$,
cf.~Example~\ref{e:regular} in the regular case.
\end{ex}

For $(f,\wt)$-extremal metrics, we have noted that the weight $\wt=2m$ has a
special interpretation in conformal geometry.  The next lemma provides an
analogous interpretation in CR geometry mentioned above of the weight
$\wt=m+2$ for $(\cae,\wt)$-extremal metrics.

\begin{lemma}\label{l:main}
For any $\cae\in\crJ(\Sm,\Ds,J)$, $\sigma(\cae):=
\Scal_{\cae,m+2}(g_\sas)\,\sas\in\con(\Sm,\Ds)$ is independent of
$\sas\in\crJ_+^\cae$.  Hence $\cae\mapsto\sigma(\cae)$ is a second order
quadratic differential operator, with $\sigma(\cae)=\Scal(g_\cae)\,\cae$ in
the case that $\cae\in \crJ_+(\Sm,\Ds,J)$.
\end{lemma}

We now emphasise a key feature of Sasaki geometry, which was used in
\cite{CFO,FOW,legendre2,MSY2} to construct CSC Sasaki manifolds from
K\"ahler manifolds which are not necessarily CSC (see
also~\cite{BHLT,FO0}). Namely, $\crJ_+(\Sm, \Ds, J)$ is open in
$\crJ(\Sm, \Ds, J)$, so if $(\Sm, \Ds, J)$ is of Sasaki type, and
$\dim \crJ(\Sm, \Ds, J)\geq 2$, we obtain a family of Sasaki
structures $\sas$ on $\Sm$, inducing transversal K\"ahler structures
on $(\Ds, J)$. With this in mind, the following is our main result,
which is an immediate consequence of Lemma~\ref{l:main}, but has many
ramifications.

\begin{thm}\label{thm:main} Let $(\Sm, \Ds, J)$ be a CR $(2m+1)$-manifold
with Sasaki cone $\crJ_+(\Sm, \Ds, J)$. Then, for any $\sas, \cae \in
\crJ_+(\Sm, \Ds, J)$ with $[\sas,\cae]=0$, $(\Sm, \Ds, J, \sas)$ is $(\cae,
m+2)$-extremal if and only if $(\Sm, \Ds, J, \cae)$ is extremal.
\end{thm}

Together with Example~\ref{e:regular-(xi,wt)-extremal}, this result shows
that the constructions of $(f, m+2)$-extremal K\"ahler metrics available in
\cite{ambitoric1,LeB1,LeB2,FO0,KTF,lahdili1} yield many new extremal Sasaki
metrics.

We further observe that since $\cL_{X_\cae} (\Scal(g_\cae))=0$,
$[\cae,\sigma(\cae)]=0$.  Hence the equivalent conclusions of
Theorem~\ref{thm:main} correspond to the occurrence of $\sigma(\cae)$ in
$\crJ^\cae$; we then have $\pm\sigma(\cae) \in \crJ_+^\cae$ if and only if the
scalar curvature $\Scal(g_\cae)$ is everywhere positive or everywhere
negative, in which case we can take $\sas=\pm\sigma(\cae)$ in
Theorem~\ref{thm:main} to obtain $\Scal_{\cae,m+2}(g_\sas)=\pm 1$.

\begin{cor}\label{cor:csc} Let $(\Sm, \Ds, J,\cae)$ be an extremal
Sasaki $(2m+1)$-manifold. Then the $\crJ_+^\sas$-family of $(\cae,
m+2)$-extremal Sasaki structures on $(\Sm, \Ds, J)$ contains a Sasaki
structure $\sas:=\pm\sigma(\cae)$ of constant nonzero $(\cae, m+2)$ scalar
curvature if and only if the extremal Sasaki structure $\cae$ has nowhere zero
scalar curvature. Thus there is an equivalence between Sasaki manifolds
$(\Sm,\Ds,J,\sas)$ of constant nonzero $(\cae, m+2)$ scalar curvature and
extremal Sasaki manifolds $(\Sm, \Ds, J,\cae)$ with nowhere zero scalar
curvature.
\end{cor}

The proofs of Lemmas~\ref{l:contact}--\ref{l:main} are straightforward
rephrasings of standard results in CR geometry, but for the convenience of the
reader, we indicate their proofs in Section~\ref{s:background}.  In Section~\ref{s:GIT}, we define, on a
compact contact manifold of Sasaki type, a formal GIT setting for the search
for $(\cae, \wt)$-extremal Sasaki structures, extending the picture in
\cite{He} and providing a conceptual explanation for the key Lemma~\ref{l:main} above.  Then in the
rest of the paper we return to K\"ahler geometry and applications of
Theorem~\ref{thm:main}, which gives a way of relating different K\"ahler
geometries locally or (under suitable rationality conditions) globally.  We
formalize this as follows.

\begin{defn}\label{CR f-twist} Let $(M, g,J, \omega)$ be a K\"ahler manifold
and $f$ a positive Killing potential. We say that $(\tilde M, \tilde g, \tilde
J,\tilde \omega)$ is a \emph{CR twist} of $M$ by $f$, or an \emph{$f$-twist}
for short, if it is a Sasaki--Reeb quotient by the Sasaki structure $f\sas$ on
the Sasaki manifold $(\Sm, \Ds ,J, \sas)$ corresponding (over any open subset
where $[\omega/2\pi]$ is integral) to $M$ via Example~\ref{e:regular}.
\end{defn}
A CR $f$-twist can be seen as a special case of the twist construction of Swann \cite{swann} (see also \cite{joyce,MMW}) which has been used to study different geometric structures. In these terms, Theorem~\ref{thm:main} shows that any extremal K\"ahler metric
can (locally) be obtained from a $(f, m+2)$-extremal K\"ahler metric, via a CR
$f$-twist, while Corollary~\ref{cor:csc} establishes an equivalence between
K\"ahler metrics of constant $(f, m+2)$ scalar curvature and extremal K\"ahler
metrics of nonvanishing scalar curvature.

In real dimension $2m=4$, the latter reduces to the equivalence between
conformally Einstein--Maxwell K\"ahler metrics $(g, J, \omega, f)$ and
extremal metrics $(\tilde g, \tilde J, \tilde \omega)$ of nonvanishing scalar
curvature alluded to above: the extremal K\"ahler $4$-manifold is obtained as
the Sasaki--Reeb quotient with respect to an extremal Sasaki structure $\cae$
of $(\Sm, \Ds, J)$, whereas $(g, J, \omega, f)$ is obtained as the quotient
with respect to the Sasaki structure defined by the scalar curvature of
$\cae$. Thus our correspondence gives a conceptual explanation and
generalization of \cite[Prop.~3]{AM}, and can be used to obtain new examples
of extremal Sasaki and K\"ahler metrics from the known conformally
Einstein--Maxwell K\"ahler ones.

More generally, as any CR twist of an extremal K\"ahler metric is $(f,
m+2)$-extremal for some $f$, one of the main theses of this paper is that one
can reduce the search for extremal K\"ahler metrics to the search of $(f,
m+2)$-extremal K\"ahler metrics on simpler K\"ahler manifolds. We explore this
idea in the remainder of the paper.

\smallskip
As a warm-up, in Section~\ref{s:examples}, we consider the simplest examples:
the Bochner-flat Sasaki--Reeb quotients of CR spheres~\cite{Bryant,Webster}
and products. Then, in Section~\ref{s:toric}, we turn our attention to toric
geometry. While toric K\"ahler and Sasaki geometries have been well-studied,
to apply our theory, we develop a CR-invariant viewpoint, building
on~\cite{legendre1,legendre2,Lerman,MSY1}. In the toric case, CR-invariance
corresponds to projective invariance on the image of the momentum map, and as
an interesting side benefit, we give a manifestly projectively invariant
treatment of the Legendre transform, by relating it to a particular case of a
Bernstein--Gelfand--Gelfand resolution~\cite{BGG,Lep}, as constructed
in~\cite{CD,CSS}.  Returning the main line of the paper, we then obtain
explicit descriptions of the CR $f$-twists of toric manifolds and of toric
bundles given by the generalized Calabi ansatz, showing that the latter are CR
$f$-twists of a product metric. In Section~\ref{s:ansatz}, we recast, in terms
of the general correspondence herein, some of the explicit families of $(f,
m+2)$-extremal K\"ahler metrics, including those obtained by the regular
ambitoric ansatz in \cite{ambitoric1} and by an ansatz in \cite{ACGL}. This
leads both to a higher dimensional extension of the regular ambitoric
ansatz~\cite{ambitoric1} and to a complete classification of the $(f,
m+2)$-extremal K\"ahler metrics obtained by this ansatz. In the final
Section~\ref{s:global}, we turn to global considerations. We define the Calabi problem for $(\cae, \wt)$-extremal
Sasaki metrics, which naturally generalizes the existence problem of extremal
Sasaki metrics in a given Sasaki polarization \cite{BGS}, recently studied in
many places~\cite{BT0,BT,BV,CSz,legendre2,MSY1,MSY2,V}.  We end by
illustrating how the (non)existence of $(f,\wt)$-extremal K\"ahler metrics in
a given integral K\"ahler class of a geometrically ruled complex surface (as
studied in \cite{AMT,KTF,lahdili2}) leads both to existence and non-existence
results for extremal Sasaki metrics compatible with (possibly irregular)
Sasaki--Reeb vector fields on the corresponding contact manifolds. In
particular, we establish the following Yau--Tian--Donaldson type
correspondence.

\begin{thm}\label{thm:ruled} Let $(M, J)= P(\cO \oplus \cL) \to B$ be a compact
ruled complex surface over a Riemann surface $B$, $L$ a polarization of
$(M,J)$, and $\omega\in 2\pi c_1(L)$ an $\Sph^1$-invariant K\"ahler metric
with respect to the circle action by scalar multiplication in $\cO$.  Let
$(\Sm, \Ds, J, \sas)$ be the regular Sasaki manifold over $(M, J, \omega)$
given by Example~\textup{\ref{e:regular}}, $\cae \in \crJ_+(\Sm, \Ds, J)$ the
lift of the generator of the $\Sph^1$-action on $M$ and $\hat Z_\cae$ the
induced holomorphic vector field on $L$.  Then $\Sm$ admits a
$X_\sas$-invariant, $\Ds$-compatible CR structure which is extremal Sasaki
with respect to $\cae$ if and only if $(M, L, \hat Z_\cae)$ is analytically
relatively $(\hat Z_\cae, 4)$ \textup K-stable with respect to admissible
test-configurations in the sense of \cite{AMT}.
\end{thm}

This suggests a link between the weighted K-stability of \cite{AMT,lahdili2}
for a smooth polarized variety, and K-stability of the K\"ahler cone of a
Sasaki polarization, studied in \cite{BV,CSz}.

\section{Proofs of Lemmas~\ref{l:contact}--\ref{l:main}}
\label{s:background}

Let $(\Sm,\Ds)$ be a contact $(2m+1)$-manifold, i.e., $\Ds\leq T\Sm$ is a
rank $2n$ distribution on $\Sm$, with quotient map $\cf_\Ds\colon T\Sm\to
T\Sm/\Ds$, whose Levi form $\Lv_\Ds(X,Y)=-\cf_\Ds(X,Y)$ ($X,Y\in
C^\infty_\Sm(\Ds)$) is nondegenerate at each point of $\Sm$. A \emph{CR
  structure} on $(\Sm,\Ds)$ is a complex structure $J$ on $\Ds$ such that the
subbundle $\Ds^{(1,0)}$ of $(1,0)$-vectors in $\Ds\otimes\C$ is closed under
Lie bracket. This implies in particular that $J$ is an \emph{almost CR
  structure}, i.e., a complex structure on $\Ds$ such that $\Lv_\Ds$ has
type $(1,1)$ with respect to $J$. We say that $(\Ds, J)$ is \emph{strictly
  pseudo-convex} if $\Lv_\Ds$ is definite (with respect to $J$) at each point
of $\Sm$, i.e., $\Lv_\Ds(\cdot,J\cdot)$ is a definite bundle metric on $\Ds$.
It then follows that $(\Sm,\Ds)$ is \emph{co-oriented}, i.e., $T\Sm/\Ds$ is
an oriented real line bundle.

\begin{proof}[Proof of Lemma~\textup{\ref{l:contact}}] We show that any
section $\cae$ of $T\Sm/\Ds$ has a unique lift to a contact vector field
$X_\cae$ (with $\cf_\Ds(X_\cae)=\cae$). On the open subset where $\cae$ is
nonzero, the contact condition implies that $X_\cae$ is the Reeb vector field
of the contact form $\cf_\cae:=\cae^{-1}\cf_\Ds$, which is characterized by
$\cf_\cae(X_\cae)=1$ and $\d\cf_\cae(X_\cae,\cdot)=0$. Now suppose
$\cae=f\sas$ with $\sas$ nonvanishing and $f$ a smooth function. Then
$\d\cf_\sas=\d f\wedge \cf_\cae + \d\cf_\cae$ and the characterization of Reeb
vector fields gives
\begin{equation}\label{eq:Reeb-change}
  X_\cae = f X_\sas - (\d\cf_\sas\restr\Ds)^{-1}(\d f\restr\Ds)
\end{equation}
where $f$ is nonzero, but this formula extends $X_\cae$ smoothly over the
zeroset of $f$. This also shows $\cae\mapsto X_\cae$ is a first order
differential operator. Since a contact vector field in $\Ds$ is necessarily
zero by the nondegeneracy of the Levi form, the lift is unique.
\end{proof}

\begin{proof}[Proof of Lemma~\textup{\ref{l:potential}}]
Clearly $\cL_{X_\sas}(f\sas)=\d f(X_\sas)\sas$, and if $f,h\in
C^\infty_M(\R)^\sas$ then~\eqref{eq:Reeb-change} implies
\[
[f\sas,h\sas]=\d h(X_{f\sas})
=-(\omega_\sas\restr\Ds)^{-1}(\d f\restr\Ds,\d h\restr\Ds)
\]
so the first part is immediate. We may thus suppose that $f$ is the pullback
of a smooth function, also denoted $f$, on a Sasaki--Reeb quotient
$(M,g,J,\omega)$, with symplectic gradient $K$. Then the same
formula~\eqref{eq:Reeb-change} shows that $X_\cae$ is a lift of $K$ to
$\Sm$. Furthermore, since $X_\cae$ is contact and $X_\sas$-invariant,
$\cL_{X_\cae} J$ is horizontal and $X_\sas$-invariant, hence vanishes iff its
pushforward to $M$ vanishes, which holds iff $\cL_K J =0$ on $M$, i.e., $f$ is
a Killing potential.
\end{proof}

\begin{proof}[Proof of Lemma~\textup{\ref{l:main}}] We let $\cae=f\sas$
  with $\sas\in \crJ_+^\cae$ and expand the definition of the $(\cae, m+2)$
  scalar curvature of $g_\sas$ on $\Sm$ to obtain
\begin{align}\label{eq:invariant1}
\Scal_{\cae,m+2}(g_\sas)\,\sas &= \bigl(f^2 \Scal(g_\sas) - 2(m+1) f\Delta_{g_\sas} f
-(m+1)(m+2)\bigl|\d f\restr\Ds\bigr|{}^2_{g_\sas}\bigr)\,\sas\\
\label{eq:invariant2}
&=\Bigl(f \Scal(g_\sas) - 2(m+1) \Delta_{g_\sas} f
-\frac{(m+1)(m+2)}{f}\bigl|\d f\restr\Ds\bigr|{}^2_{g_\sas} \Bigr)\,\cae.
\end{align}
where~\eqref{eq:invariant2} holds on the open subset $U$ where $\cae$ is
nonzero. Now, as noted already, $\Scal(g_\cae)$ and $\Scal(g_\sas)$ are the
Tanaka--Webster scalar curvatures of the Tanaka--Webster connections induced
by $\cae$ and $\sas$~\cite[\S4]{david:weyltanaka}, and a straightforward but
tedious computation of the change of the Tanaka--Webster scalar curvature
under a change of connection, which can be found e.g.~in~\cite[(2.9)]{JL},
shows that on $U$,~\eqref{eq:invariant2} computes $\Scal(g_\cae)\cae$,
independently of $f$. However, on any open subset where $\cae=0$, $f=0$ and
hence the right hand side of~\eqref{eq:invariant1} is zero. Thus
$\Scal_{\cae,m+2}(g_\sas)\,\sas$ is independent of $\sas$ on a dense open
subset, hence everywhere. The equality~\eqref{eq:invariant1} now shows that
this is a second order quadratic differential operator in $\cae$.
\end{proof}

\section{Formal GIT picture for weighted extremal Sasaki metrics}\label{s:GIT}

Let $(\Sm,\Ds)$ be a compact co-oriented contact $(2m+1)$-manifold (or
orbifold), and fix a torus $\T$ in its group $\Con(\Sm,\Ds)$ of contact
transformations.  As explained in the introduction, we tacitly identify the
Lie algebra of $\Con(\Sm, \Ds)$ with the space $\con(\Sm, \Ds)$ of smooth
sections of $T\Sm/\Ds$ and denote by $\con_+(\Sm, \Ds)$ the open cone of
positive sections in of $T\Sm/\Ds$ with respect to its orientation. Let
$\Con(\Sm,\Ds)^\T$ denote the group of $\T$-equivariant contact
transformations, with Lie algebra identified with the space $\con(\Sm,\Ds)^\T$
of $\T$-invariant sections of $T\Sm/\Ds$. Thus, the Lie algebra $\tor$ of $\T$
as a linear subspace of $\con(\Sm,\Ds)^\T$.

Now observe that $\vol_\Ds :=\cf_\Ds\wedge\Lv_\Ds^{\wedge m}$ is a
well-defined section of $\Wedge^{2m+1}T^*\Sm\otimes (T\Sm/\Ds)^{m+1}$: indeed
for any nonvanishing section $\sas$ of $T\Sm/\Ds$, $\sas^{-m-1}\vol_\Ds
=\cf_\sas\wedge \d\cf_\sas^{\wedge m}$. Fix two such sections
$\sas,\cae\in\tor \cap \con_+(\Sm, \Ds)$ and $\wt\in\R$. Then
$\con(\Sm,\Ds)^\T$ has a bi-invariant inner product
\begin{equation*}
\ip{ \cae_1, \cae_2 }_{\cae,\sas,\wt} :=
\int_\Sm (\cae_1/\sas) (\cae_2/\sas) (\cae/\sas)^{-\wt-1}
\,\cf_\sas\wedge \d\cf_\sas^{\wedge m}
=\int_\Sm \cae_1\,\cae_2\,\cae^{-\wt-1}\sas^{\wt-2-m}\,\vol_\Ds.
\end{equation*}

Let $\mathcal{AC}_{+}(\Sm,\Ds)^\T$ be the space of $\T$-invariant almost
CR  structures on $(\Sm,\Ds)$, such that $\Lv_{\Ds}$ is of type $(1,1)$ and positive definite
with respect to $J$ and the given orientation on $T\Sm/\Ds$. We denote by
${\mathcal C}_+(\Sm, \Ds)^{\T} \sub \mathcal{AC}_{+}(\Sm,\Ds)^\T$ the subset
of $\T$-invariant compatible CR structures on $(\Sm, \Ds)$.  Notice that
$\Con(\Sm,\Ds)^\T$ acts naturally on $\mathcal{AC}_+(\Sm,\Ds)^\T$ (preserving
${\mathcal C}_+(\Sm, \Ds)^{\T}$) and the tangent space of
$\mathcal{AC}_+(\Sm,\Ds)^\T$ at $J$ is identified with the Fr\'echet space of
smooth sections $\dot J$ of $\End(\Ds)$ satisfying
\begin{equation*}
\dot{J} J + J \dot{J} =0, \qquad \Lv_\Ds(\dot{J} \cdot , \cdot) +
\Lv_\Ds(\cdot, \dot{J} \cdot) =0,
\end{equation*}
so $\mathcal{AC}_+(\Sm,\Ds)^\T$ has a formal Fr\'echet K\"ahler structure
$(\J, \Om^{\cae,\sas,\wt})$ defined by $\J_{J} (\dot J) := J \dot J$ and
\begin{equation*}
\Om^{\cae,\sas,\wt}_{J}(\dot{J}_1, \dot{J}_2)
:= \frac{1}{2} \int_\Sm \trace \bigl(J\dot{J}_1 \dot{J}_2\bigr)\,
(\cae/\sas)^{-\wt+1}\, \cf_\sas\wedge \d\cf_\sas^{\wedge m}
= \frac{1}{2} \int_\Sm \trace \bigl(J\dot{J}_1 \dot{J}_2\bigr)\,
\cae^{-\wt+1} \sas^{\wt-2-m}\,\vol_\Ds.
\end{equation*}
To see this, we can take $\sas \in \tor \cap \con_+(\Sm, \Ds)$ to be
quasi-regular with a global quotient $(M,\omega)$. Then our set-up reduces to the
formal GIT picture for $(f,\wt)$-extremal $\omega$-compatible, $\T/\Sph^1_{\sas}$-invariant almost-K\"ahler metrics on the symplectic orbifold
$(M, \omega)$, discussed in~\cite{ACGL,AM,lahdili1}. The momentum map for the
action of $\Ham(M,\omega)^\T$ at a compatible K\"ahler structure is identified
with the $(f,\wt)$ scalar curvature, showing that for a CR structure $J\in
\mathcal{C}_{+}(\Sm,\Ds)^\T$, the corresponding momentum map for the action of
$\Con(\Sm,\Ds)^\T$ on $\mathcal{AC}_{+}(\Sm,\Ds)^\T$ is $\Scal(g_\cae)\cae$ (where we multiply by $\cae$ to
obtain an element of $\con(\Sm,\Ds)^\T$).

We now notice that for $\wt=m+2$, the bi-invariant inner product on $\con(\Sm,
\Ds)^{\T}$ and the formal K\"ahler structure on $\mathcal{AC}_{+}(\Sm,\Ds)^\T$
are independent of $\sas$. In this case, our setting reduces to the formal GIT
picture for extremal Sasaki metrics on $(\Sm,\cf_\cae)$ discussed
in~\cite{He}, where the momentum map for the action of $\Con(\Sm,\Ds)^\T$ is 
the Tanaka--Webster scalar curvature $\Scal(g_\cae)\cae$  (the
multiplication by $\cae$ is implicit in~\cite{He} through the identification
of the Lie algebra with smooth functions).

Hence this provides another explanation as to why the weight $m+2$ is special
and the transversal $(\cae,m+2)$ scalar curvature
$\Scal_{\cae,m+2}(g_\sas)\sas$ of $g_\sas$ is independent of $\sas$ and equal
to the Tanaka--Webster scalar curvature of $g_\cae$, viewed as an element of
$\con(\Sm,\Ds)^\T$.

\section{Basic Examples}\label{s:examples}

\subsection{Bochner-flat $(f, m+2)$-extremal metrics} Let us now consider the 
\emph{standard CR sphere} $\Sph^{2m+1} \sub \C^{m+1}$, $m \geq 2$, with $\Ds =
T\Sph^{2m+1} \cap J (T\Sph^{2m+1})$, $J$ induced by the standard complex
structure on $\C^{m+1}$, and $\crJ(\Sph^{2m+1}, \Ds, J)\cong \su(1,
m+1)$. By a result of Webster~\cite{Webster}, for any $\sas \in
\crJ_+(\Sph^{2m+1}, \Ds, J)$, the transversal K\"ahler structure
$(g_\sas,J,\omega_\sas)$ is Bochner-flat, and thus extremal
(see~\cite{Bryant}), and any Bochner-flat K\"ahler manifold $(M,g, J, \omega)$
is (locally) obtained as a $\sas$-reduction of $(\Sph^{2m+1}, \Ds, J)$ for
some such $\sas$.  It then follows from Theorem~\ref{thm:main} that for any
$\cae\in \crJ_+^\sas$, $(\Ds, J, \sas)$ is a $(\cae, m+2)$-extremal, and hence
by Lemma~\ref{l:potential} (see Example~\ref{e:regular-(xi,wt)-extremal}), we
have the following observation.

\begin{prop}\label{BF-local} Let $(M, g, J, \omega)$ be a Bochner-flat K\"ahler
$2m$-manifold and $f>0$ a Killing potential. Then $(g, \omega)$ is $(f,
  m+2)$-extremal.
\end{prop}

To obtain global examples, we let $\sas_\wv \in \crJ_+(\Sph^{2m+1}, \Ds, J)$
correspond to the weighted Hopf fibration $\Sph^{2m+1} \to \C P_\wv^m$,
realizing $(\Sph^{2m+1}, \Ds, J, \sas_\wv)$ as a quasi-regular Sasaki manifold
over the Bochner-flat weighted projective space $(\C P_\wv^m, J, g, \omega)$
(see~\cite{Bryant,DG}). Thus we have the following higher dimensional
extension of \cite[Prop. 5]{AM}.

\begin{cor}\label{c:FS} The Bochner-flat metric on $\C P_\wv^m$ is
$(f, m+2)$-extremal for any positive Killing potential $f$.
\end{cor}

\subsection{Flat $(f,\wt)$-extremal metrics} The conclusion of
Proposition~\ref{BF-local} can be strengthened for flat K\"ahler metrics.

\begin{prop}\label{flat-local}
Let $(V, g_V, \omega_V)$ be a flat K\"ahler manifold and $f>0$ a Killing
potential on $V$. Then, for any scalar-flat K\"ahler manifold $(\Bm, g_\Bm,
\omega_\Bm)$ the K\"ahler product $(M, g, \omega)$ of $(V, g_V, \omega_V)$ and
$(\Bm, g_\Bm, \omega_\Bm)$ is $(f,\wt)$-extremal for any $\wt$.
\end{prop}

\begin{proof} As $\Bm$ is scalar-flat,~\eqref{(f,wt)-scalar-curvature} implies
that the $(f,\wt)$ scalar curvature of $V\times \Bm$ equals the $(f,\wt)$
scalar curvature of $M$. Thus, we need to establish the claim on $M:=V$. As
$g$ is a flat metric, \eqref{(f,wt)-scalar-curvature} reduces to
\begin{equation}\label{flat}
-2(\wt-1)f\Delta_g f -\wt(\wt-1)|\d f|^2_g
\end{equation}
so it suffices to show that each of the two terms in \eqref{flat} is a Killing
potential for $g$. For the first term, using that $g$ is Ricci-flat and $f$ is
a Killing potential, the Bochner identity shows that $\Delta_g f$ is a
constant, and thus $f\Delta_g f$ is a Killing potential of $g$. For the second
term, using that $g$ is Bochner-flat and Proposition~\ref{BF-local}, it
follows that \eqref{flat} with $\wt=m+2$ gives rise to a Killing potential,
and hence $|\d f|^2_g$ is a Killing potential for $g$.
\end{proof}

\subsection{$(f, m+2)$-extremal products} As noted in the proof of
Proposition~\ref{flat-local}, the K\"ahler product of a scalar-flat, K\"ahler
$2(m-\ell)$-manifold $(\Bm, g_\Bm, \omega_\Bm)$ with a $(f,\wt)$-extremal
K\"ahler $2\ell$-manifold $(V, g_V, \omega_V)$ gives rise to a
$(f,\wt)$-extremal $2m$-manifold $(M, g, \omega)$.  In \cite{AMT, lahdili1},
for any given $\wt$, large families of $(f,\wt)$-extremal Hodge (i.e.,
compact, integral) K\"ahler manifolds $(V, g_V, \omega_V)$ of dimension
$2\ell$ are constructed. Taking such a $(V, g_V, \omega_V, f)$ with $\wt=m+2$
($m\geq \ell$) and considering the K\"ahler product of $(V, g_V, \omega_V)$
with a scalar-flat Hodge K\"ahler $2(m-\ell)$-manifold $(\Bm, g_\Bm,
\omega_\Bm)$, we obtain a compact $(f, m+2)$-extremal K\"ahler $2m$-manifold
$(M, g, \omega, J)$, which gives rise to a compact extremal Sasaki
$(2m+1)$-manifold $(\Sm, \Ds, J, \cae)$ via Example~\ref{e:regular}.  Notice
that the extremal Sasaki manifold thus obtained is not in general
quasi-regular, but when it is (which places a rationality condition on the
positive Killing potential $f$ of $g_V$), the resulting extremal K\"ahler
orbifold is not in general a product, even though $(M, g, \omega)$ is. We
detail and generalize this observation below in the setting of toric bundles.

\section{Toric geometry and toric bundles}\label{s:toric}

\subsection{Toric contact manifolds}
\label{s:projective-toric}

Applications of Theorem~\ref{thm:main} depend in particular on the existence
of independent commuting elements $\cae,\sas\in\crJ(\Sm,\Ds,J)\leq
\con(\Sm,\Ds)$. The maximal dimension of an abelian subalgebra of
$\con(\Sm,\Ds)$ is $m+1$ (assuming $\Sm$ is connected of dimension $2m+1$).
Let us therefore consider the case that we have such an $(m+1)$-dimensional
abelian subalgebra $\torh\into\con(\Sm,\Ds);\ah\mapsto\cae_\ah$, in which case
$(\Sm,\Ds,\torh)$ is said to be \emph{toric}. We usually assume that the
corresponding contact vector fields generate an effective contact action of a
real $(m+1)$-torus $\T^{m+1}$, whose Lie algebra is thus canonically
isomorphic to $\torh$. Hence we have an integral lattice $\Lam\sub \torh$ with
$\T^{m+1}\cong\torh/2\pi\Lam$, and, on the dense open set $\Sm\op$ where the
$\T^{m+1}$-action is free, angle coordinates $t\colon\Sm\op\to\torh/2\pi\Lam$.

We also assume that the tautological bundle homomorphism 
\begin{align*}
\Sm\times \torh&\to T\Sm/\Ds\\
 (\pt,\ah)&\mapsto \cae_\ah(\pt)
\end{align*}
is surjective (as it is in the Sasaki case), so that its transpose
$(T\Sm/\Ds)^*\to \Sm\times\torh^*$ is injective.  We thus obtain a
\emph{momentum map} $\bar\mm\colon \Sm\to\Proj(\torh^*)$, where $\bar\mm(\pt)$
is the image of $(T\Sm/\Ds)^*_\pt$ in $\torh^*$, for any $\pt\in \Sm$. Hence
$\bar\mm^*\cO_{\torh^*}(-1)\cong (T\Sm/\Ds)^*$, where $\cO_{\torh^*}(-1)$ is
the tautological line bundle over the projective space $\Proj(\torh^*)$, with
fibre $\cO_{\torh^*}(-1)_\tau=\tau\leq \torh^*$.

\begin{rem} Alternatively observe (cf.~\cite{Lerman}) that the
annihilator $\Ds^0\leq T^*\Sm$ of $\Ds$ inherits from $T^*\Sm$ a closed
$2$-form which is nondegenerate on the complement of the zero section, and any
contact vector field on $\Sm$ lifts to a hamiltonian vector field on
$\Ds^0$. The restriction to $\torh$ of the momentum map of this action is
$\tilde\mm\colon \Ds^0\to \torh^*$ with
$\ip{\tilde\mm(\alpha),\ah}=\alpha(\cae_\ah)$ for $\alpha\in \Ds^0$ and
$\ah\in \torh$, where angle brackets denote contraction of $\torh^*$ with
$\torh$. Using the natural duality $\Ds^0\cong (T\Sm/\Ds)^*$,
$(\pt,\tilde\mm(\alpha))$ is the element of $\bar\mm^*\cO_{\torh^*}(-1)\leq
\Sm\times \torh^*$ corresponding to $\alpha\in\Ds^0_\pt$.
\end{rem}

Herein, we generally work instead with the \emph{momentum section}
$\hat\mm\colon\Sm\to \torh^*\otimes(T\Sm/\Ds)$ defined by
$\ip{\ah,\hat\mm(\pt)}=\cae_\ah(\pt)\in (T\Sm/\Ds)_\pt$ for $\ah\in\torh$ and
$\pt\in\Sm$. If $\mc\colon\Proj(\torh^*)\to\torh^*\otimes\cO_{\torh^*}(1)$
denotes the tautological section, with $\cO_{\torh^*}(1):=\cO_{\torh^*}(-1)^*$
and $\ip{\ah,\mc(\tau)}=\ip{\ah,\cdot}\restr\tau$ for $\ah\in\torh$ and
$\tau\in \Proj(\torh^*)$, then under the isomorphism
$\bar\mm^*\cO_{\torh^*}(1)\cong T\Sm/\Ds$, $\hat\mm=\bar\mm^* \mc$.

Any nonzero $\eps\in \torh$ defines an affine chart
\[
\As:= \{\pt\in\torh^*|\ip{\eps,\pt}=1\} \into\Proj(\torh^*)
\]
and if $\Dm\sub\Proj(\torh^*)$ is an open subset of the image of this chart,
then $\ip{\eps,\mc}$ restricts to a trivialization of
$\cO_{\torh^*}(1)\restr\Dm$.  In this trivialization, $\mc$ is the affine lift
of $\Dm$ to $\As\sub\torh^*$, i.e., we have $\mc\colon \Dm\to\torh^*$ with
$\ip{\eps,\mc}=1$. Hence $\ip{\eps,\d\mc}=0$, i.e.,
\[
\d\mc\colon TU\to U\times\eps^0,\qquad\text{where}\qquad
\eps^0=\{\pt\in\torh^*|\ip{\eps,\pt}=0\},
\]
is the trivialization of $TU$ in this affine chart. Thus $\torh$ is naturally
identified with the space $\Aff(\Dm)$ of affine functions on $\Dm$:
$\ah\in\torh$ defines the affine function $\ip{\ah,\mc}$ on $\Dm$, with $\eps$
corresponding to the constant function $1$. If $\tor=\torh/\spn{\eps}$, there
is a short exact sequence
\begin{equation}\label{eq:fund-seq}
0 \to \R  \xrightarrow{\eps} \torh \xrightarrow{\delta} \tor \to 0,
\end{equation}
where the duality $\tor^*\cong\eps^0$ identifies $\tor$ with $T^*_\pt \Dm$
for any $\pt\in\Dm$, while the quotient map $\delta$ sends an affine function
to its linear part (the constant value of its derivative).

Now $\bar\mm\colon\Sm\to\Proj(\torh^*)$ takes values in the affine chart
defined by $\eps$ iff $\sas:=\cae_\eps\in \con(\Sm,\Ds)$ is a nonvanishing
section of $T\Sm/\Ds$. In this trivialization the momentum section becomes a
function $\hat\mm\colon\Sm\to \torh^*$ with $\ip{\eps,\hat\mm}=1$.  The
transversal K\"ahler form $\d\cf_\sas\restr\Ds$ descends to a symplectic form
$\omega$ on any quotient of $M$ of $\Sm$ by $\sas$. If $\sas$ is
quasi-regular, we may take $(M,\omega)$ to be the global quotient, which is
then a toric symplectic $2m$-orbifold under the hamiltonian action of the real
$m$-torus $\T^m=\T^{m+1}/\Sph^1_\sas$ where $\Sph^1_\sas$ is the circle action
generated by $X_\sas$.

Concretely, we can choose a basis $e_0,\ldots e_m$ for $\torh$ such that
$\eps=e_0$, introduce coordinates $\mc_j=\ip{e_j,\mc}$ on $U$ with $\mc_0=1$,
and write $\ip{\ah,\mc}=\ah_0+\ah_1 \mc_1+\cdots \ah_n \mc_n$. We then
have coordinates $\hat\mm=(\hat\mm_0, \hat\mm_1, \ldots \hat\mm_m)$ and
$t=(t_0, t_1, \ldots t_m)$ on $\Sm\op$ with $\hat\mm_0=1$ and
\begin{equation}\label{eq:toric-contact}
\cf_\sas =\ip{\hat\mm,\d t} = \d t_0 + \sum_{j=1}^m \hat\mm_j \d t_j. 
\end{equation}
Hence $\d\cf_\sas= \ip{\d\hat\mm\wedge \d t}$, and since $\d\hat\mm$ takes
values in $\tor^*$, $\ip{\d\hat\mm\wedge \d t}$ depends only on $\delta(\d
t)$, which descends to a $\tor$-valued $1$-form on $M$. If we denote by $\mm$
the map $M\to\As$ induced by the momentum section on $\Sm$ in the affine chart
$\As$, we may therefore write $\omega = \ip{\d\mm\wedge \d t}$ on the image
$M\op$ of $\Sm\op$. The isomorphism $f\mapsto f\sas$ of
Lemma~\ref{l:potential} identifies $\torh$ with an abelian Lie subalgebra of
$C^\infty_M(\R)$ under the Poisson bracket induced by $\omega$. Evidently,
$\sas=\cae_\eps$ corresponds to $f\equiv 1$ on $M$. More generally, for any
$\ah\in\torh$ the function $f_\ah=f_\ah(\mm)\in C^\infty_M(\R)$ corresponding
to $\cae_\ah\in\con(\Sm,\Ds)$ satisfies $f_\ah(\mm(\pt))=\ip{\ah,\mm(\pt)}$
for all $\pt\in M$, i.e., $f_\ah(\mm)=\ip{\ah,\mm}$ is the pullback by $\mm$ of
the affine function $f_\ah(\mc)=\ip{\ah,\mc}$ on $\As$. Pulling back by $\mm$,
we may thus reinterpret~\eqref{eq:fund-seq} on $M$: in particular, we may view
$\delta$ as the restriction to $\torh\into C^\infty_M(\R)$ of the symplectic
gradient, and $\tor$ as its image in $\Ham(M,\omega)$.

\subsection{Toric CR manifolds and their Sasaki--Reeb quotients}

Thus far we have only considered the toric contact geometry of $\Sm$ and
induced toric symplectic geometry on $M$. According to~\cite{guillemin}, on
the dense open subset $M\op$, any toric almost K\"ahler structure may be
written in momentum--angle coordinates $(\mm, t)$ as:
\begin{equation}\label{toric-kahler}
\begin{aligned}
g &= \ip{\d \mm,  G(\mm), \d \mm} + \ip{\d t, H(\mm), \d t },\qquad
&  J\d t &= -\ip{G(\mm),\d \mm},\\
\omega &= \ip{ \d \mm\wedge \d t }, & J\d \mm &= \ip{H(\mm),\d t},
\end{aligned}
\end{equation}
where $H$ is a smooth positive definite $S^2\tor^*$-valued function on the
momentum image $\Delta\op:=\mm(M\op)$ and $G=H^{-1}$ is its pointwise inverse,
a smooth $S^2{\tor}$-valued function. This local expression makes sense in the
affine setting, where $\Delta\op$ lies in an affine space $\As\sub\torh^*$,
modelled on $\tor^*$: $G$ is a metric on $T\Delta\op\cong \Delta\op\times
\tor^*$, and $H$ the inverse metric on $T^*\Delta\op$.

A metric of the form \eqref{toric-kahler} is K\"ahler, i.e., $J$ is
integrable, if and only if $\ip{\d G\wedge \d \mc}=0$, which in affine
coordinates $\mc=(1,\mc_1,\ldots \mc_m)$ on $\As$ reads
\begin{equation*}
\partial G_{ij}/\partial \mc_k = \partial G_{ik}/\partial \mc_j
\end{equation*}
for all $i,j,k\in \{1,\ldots m\}$. Since $G$ is symmetric, this is the
integrability condition to write $G=\Hess(u)$ for a smooth strictly convex
function $u$ defined on the momentum image $\Delta\op$, which is called a
\emph{symplectic potential}. When $M$ is a compact manifold (or orbifold),
Delzant theory~\cite{Delzant, LT} implies that $\Delta\op$ is the interior of
a (rational) Delzant polytope $\Delta \sub \tor^*,$ and $u$ satisfies the
Abreu boundary conditions \cite{Abreu2} on $\partial \Delta$.

The theory of symplectic potentials in toric K\"ahler geometry thus relies
upon a locally exact complex of linear differential operators
\begin{equation}\label{toric-complex}
C^\infty_\As(\R)\xrightarrow{\Hess} C^\infty_\As(S^2\tor)\xrightarrow{\cD}
C^\infty_\As(\Wedge^2\tor\odot\tor)
\end{equation}
where $\cD(G)=\ip{\d G\wedge \d \mc}$ and $\Wedge^2\tor\odot\tor$ denotes the
alternating-free tensors in $\Wedge^2\tor\otimes\tor$ (the kernel of the
projection, alternation, to $\Wedge^3\tor$). This complex is invariant under
affine transformations by construction, but can actually be made projectively
invariant. To do this, observe that the kernel of the hessian consists of
affine functions, which on a domain $\Dm\sub\Proj(\torh^*)$ in projective
space are not naturally ordinary functions, but sections of
$\cO_{\torh^*}(1)$, as we discussed above.  Also the cotangent space is
naturally $T^*\Dm\cong\cO_{\torh^*}(-1)^0\otimes\cO_{\torh^*}(-1)\leq
\torh\otimes \cO_{\torh^*}(-1)$. With these modifications, we obtain a locally
exact complex of projectively invariant linear differential operators,
beginning
\begin{multline}\label{proj-complex}
0\to \torh\xrightarrow{\ip{\mc,\cdot}}
C^\infty_\Dm(\cO_{\torh^*}(1))\xrightarrow{\Hess}
C^\infty_\Dm(S^2T^*\Dm\otimes\cO_{\torh^*}(1))\xrightarrow{\cD}\\
C^\infty_\Dm(\Wedge^2T^*\Dm\odot T^*\Dm\otimes \cO_{\torh^*}(1))\rightarrow\cdots,
\end{multline}
and which reduces to~\eqref{toric-complex} in any affine chart.

The complex~\eqref{proj-complex} is a simple example of a
Bernstein--Gelfand--Gelfand resolution~\cite{BGG,Lep}. Without wishing to
dwell on the general machinery, we observe that the construction of this
resolution in~\cite{CD,CSS} gives a manifestly invariant construction of the
projective hessian. The main idea is to relate~\eqref{proj-complex} to the
$\torh$-valued de Rham complex
\begin{equation*}
0\to \torh\rightarrow C^\infty_\Dm(\torh)\xrightarrow{\d}
C^\infty_\Dm(T^*\Dm\otimes\torh)\to\cdots
\end{equation*}
using the following construction.
\begin{lemma} For any $u\in C^\infty_\Dm(\cO_{\torh^*}(1))$ there is a unique
$\cL(u)\in C^\infty_\Dm(\torh)$ with $\ip{\mc,\cL(u)}=u$ \textup(i.e.,
  $\cL(u)$ is a lift of $u$\textup) and $\ip{\mc, \d \cL(u)} =0$. Furthermore,
  in any local affine chart with $\mc_0=1$ and coordinates $\mc_1,\ldots
  \mc_m$, $\cL(u)=(u_0,u_1,\ldots u_m)$ where
\[
-u_0 = \mc_1 \frac{\partial u}{\partial \mc_1}+\cdots
+ \mc_m \frac{\partial u}{\partial \mc_m} - u
\]
is the Legendre transform of $u$, and $u_j=\partial u/\partial \mc_j$ for
$j\in\{1,\ldots m\}$.
\end{lemma}
\begin{proof} In local affine coordinates
with $\mc_0=1$, $\cL(u)=(u_0,u_1,\ldots u_m)$ is a lift of $u$ iff $u =
u_0+\mc_1 u_1+\cdots +\mc_m u_m$, and we require in addition
$0=\ip{\mc,\d\cL(u)} = \d u_0 + \mc_1 \d u_1 + \cdots + \mc_m \d u_m$.  Thus
$\d u = u_1 \d \mc_1+\cdots +u_m \d \mc_m$, forcing $u_j=\partial u/\partial
\mc_j$, which in turn determines $u_0$.
\end{proof}
\noindent This observation has some interesting consequences.
\begin{bulletlist}
\item $\cL\colon C^\infty_\Dm(\cO_{\torh^*}(1))\to C^\infty_\Dm(\torh)$ is a
  first order projectively invariant linear differential operator. Furthermore
  $\cL(u)$ is constant if and only if $u$ is an affine section of
  $\cO_{\torh^*}(1)$.
\item This differential lift $\cL(u)$ of $u$ is a ``universal Legendre
  transform'' in the sense that for any $\eps\in \torh$ and $\pt\in \torh^*$
  with $\ip{\eps,\pt}=1$, $-\ip{\cL(u),\pt}$ is the Legendre transformation of
  $u$ in the affine chart defined by $\eps$ with basepoint $\pt$ (it thus
  depends only on $\pt$, not $\eps$). Furthermore, the projection of $\cL(u)$
  onto $\tor=\torh/\spn{\eps}$ gives the conjugate coordinates (the components
  of $\d u$ in this affine chart, which depend only on $\eps$, not $\pt$).
\item Any $u\in C^\infty_\Dm(\cO_{\torh^*}(1))$ defines a congruence of affine
  hyperplanes $\{y\in\torh: \ip{\mc(\pt),y}=u(\pt)\}$ in $\torh$ parametrized
  by $\pt\in\Dm$. The lift $\cL(u)$ is the envelope of this hyperplane
  congruence (a classical view on the Legendre transformation):
  $\ip{\mc,\cL(u)}=u$ and $\ip{\mc,\d\cL(u)} = 0$.
\end{bulletlist}

To complete the construction of the projectively invariant hessian, it remains
to observe that since $\ip{\mc,\d\cL(u)}=0$, $\d\cL(u)$ is a section of
$T^*\Dm\otimes\cO_{\torh^*}(-1)^0$; hence $\ip{\d\mc\otimes\d\cL(u)}$ is a
section of $S^2T^*\Dm\otimes\cO_{\torh^*}(1)$, since $\d\mc$ is well-defined
modulo $\cO_{\torh^*}(-1)$ and we have
$0=\d\ip{\mc,\d\cL(u)}=\ip{\d\mc\wedge\d\cL(u)}$. In affine coordinates,
$\d\cL(u)=(\d u_0, \d (\partial u/\partial \mc_1),\ldots \d (\partial
u/\partial \mc_m))$ and projecting away from $\eps=(1,0,\ldots 0)$ gives the
usual hessian of $u$.

To apply this to a toric contact manifold $(\Sm,\Ds,\torh)$ with
$\bar\mm(\Sm\op)=\Delta\op$, it is convenient to view the projectively
invariant hessian of $u\in C^\infty_{\Delta\op}(\cO_{\torh^*}(1))$ as the
section $G=\Hess(u)$ of $S^2\cO_{\torh^*}(-1)^0\otimes \cO_{\torh^*}(-1)\leq
S^2\torh\otimes\cO_{\torh^*}(-1)$ with $\ip{G(\mc),\d\mc}=\d\cL(u)$. Then we
may define a CR structure $J$ on $\Sm\op$ by $J\d t\restr\Ds =
-\bar\mm^*\d\cL(u)\restr\Ds= -\ip{G(\hat\mm),\d\hat\mm}\restr\Ds$. For any
$\sas\in\crJ(\Sm,\Ds,J)$, this reduces to the toric K\"ahler structure defined
by $u$ on local Sasaki--Reeb quotients.

\begin{ex}\label{ex:tor3} If $m=1$ and $u=u(\mc_1)$ is a symplectic potential
in the affine chart $\mc=(1,\mc_1)$ then the differential lift of $u$ to
$\torh$ is $\cL(u)=(u(\mc_1)-\mc_1 u'(\mc_1), u'(\mc_1))$ with
$\ip{(1,\mc_1),\cL(u)}= u(\mc_1)$ and $\d\cL(u)=u''(\mc_1)
(-\mc_1,1)\,\d\mc_1$.  Thus $J\d t_0\restr\Ds
=u''(\mm_1)\mm_1\,\d\mm_1\restr\Ds$ and $J\d t_1\restr\Ds
=-u''(\mm_1)\,\d\mm_1\restr\Ds$, in accordance with $(\d t_0+\mm_1\d
t_1)\restr\Ds=\cf_\sas\restr\Ds=0$.
\end{ex}

\begin{rem} A key feature of our approach is that we avoid considering
compatible complex structures on the symplectic cone in $\Ds^0$ over $\Sm$:
such structures induce not only a CR structure $J$ on $\Sm$, but also a
preferred Sasaki structure $\sas$, a choice we wish to decouple. However, it
is straightforward to compare our approach with works such
as~\cite{Abreu-sasaki,legendre2,MSY2} which use the symplectic cone. First,
sections of $\cO_{\torh^*}(1)$ over $\Dm\sub\Proj(\torh^*)$ correspond
bijectively to homogeneous functions of degree $1$ on the inverse image of
$\Dm$ in $\torh^*\setminus\{0\}$ which contains the momentum image $\tilde\Dm$
of the symplectic cone.  Thus a symplectic potential in our sense induces an
ordinary function $u$ on $\tilde\Dm$, homogeneous of degree $1$. However, the
hessian of any such function is degenerate in radial directions, so does not
define a metric on the symplectic cone. To get around this, we exploit the
fact that symplectic potentials are not well-defined: for any $\ah\in\torh$,
we can add the linear form $\ip{\ah,\mc}\restr{\tilde\Dm}$ to $u$ without
changing its hessian.  Hence symplectic potentials are really elements of the
quotient of $C^\infty_{\tilde\Dm}(\R)$ by $\torh$.

To be concrete, if $f_1,\ldots f_k$ are linear forms on $\smash{\tilde\Dm}$
corresponding to $\ah_{(1)},\ldots \ah_{(k)}\in\torh$, then $u=\sum_{j=1}^k
f_j\log|f_j|$ is a function on $\tilde \Dm$ with $u(\lambda \pt)=\lambda
u(\pt)+\log|\lambda|\, \sum_{j=1}^k f_j$.  Hence it is homogeneous of degree
$1$ modulo $\torh$, but only strictly homogeneous of degree $1$ if $\sum_{j=1}^k
\ah_{(j)}=0$. Its hessian is $\tilde G=\sum_{j=1}^k \ah_{(j)}^2/f_j$ with
$\ip{\mc,\tilde G}=\sum_{j=1}^k \ah_{(j)}$, which is constant.

We can interpret this in our formalism by modifying the differential lift: for
$\ah\in\torh$ and $u\in C^\infty_{\Delta\op}(\cO_{\torh^*}(1))$, we define
$\cL_\ah(u)$ by $\ip{\mc,\cL_\ah(u)}=u$ and
$\ip{\mc,\d\cL_\ah(u)}=2\ah$. Assuming $\Hess(u)$ is nondegenerate and
$\ip{\ah,\mc}$ is nonvanishing, $\d\cL_\ah(u)$ is nondegenerate, and defines
the metric on the symplectic cone corresponding to the CR structure defined by
$u$ and the Sasaki structure $\cae_\ah$.
\end{rem}

\subsection{The CR twist of a toric manifold}

Suppose $(M,g,J,\omega)$ is given by~\eqref{toric-kahler} and $(\Sm, \Ds, J,
\sas)$ is a (local) Sasaki $(2m+1)$-manifold over $M$ corresponding to an
extension~\eqref{eq:fund-seq} of $\tor$ by $\R$.  We can suppose we are in the
affine picture with $\ip{\eps,\mm}=1$ on $M$, where $\sas=\cae_\eps$, and
introduce coordinates $\mc_j=\ip{e_j,\mc}$ with $e_0=\eps$ so that the induced
contact form on $\Sm$ is given by~\eqref{eq:toric-contact}, where $\hat\mm$ is
the pullback of $\mm$ to $\Sm$.

For any $\ah\in \torh$, the affine function $f_\ah(\mc)=\sum_{j=0}^m\ah_j
\mc_j$ is positive on the momentum image of $M$ if and only if $f_\ah(\mm)$ is
a positive Killing potential on $(M,g,J,\omega)$ if and only if
$\cae_\ah=f_\ah(\hat\mm)\sas$ is a Sasaki structure on $(\Sm, \Ds, J)$. A CR
$f_\ah$-twist of $(M,g,J,\omega)$ is then the induced toric K\"ahler metric on
any Sasaki--Reeb quotient of $(\Sm, \Ds, J)$ by $\cae_\ah$, which in turn has
the form \eqref{toric-kahler} for suitable coordinates $\tilde \mm, \tilde t$
and a symplectic potential $\tilde u(\tilde \mm)$ with $\tilde G=\Hess(\tilde
u)$. The new affine chart $\tilde\As$ on $\Proj(\torh^*)$ has tautological
affine coordinate $\tilde\mc$ with
$1=\ip{\ah,\tilde\mc}=f_\ah(\mc)\ip{\eps,\tilde\mc}$ and hence
$\tilde\mc=\mc/f_\ah(\mc)$.

To obtain explicit momentum-angle coordinates on $\tilde M$, we need to extend
$\tilde e_0:=a$ to a basis $\tilde e_0,\tilde e_1,\ldots \tilde e_m$ of
$\torh$.  One approach (see e.g.~\cite{legendre2}) is to assume $\ah_0\neq 0$
(which we can arrange by a translation of $\mc$) so that we may take $\tilde
e_j=e_j$ for $j\in\{1,\ldots m\}$. Since $u$ transforms as a section of
$\cO_{\torh^*}(1)$, we have the following result.

\begin{lemma}\label{l:toric-twist}
Any CR $f_\ah$-twist $(\tilde g, \tilde \omega)$ of the toric K\"ahler metric
\eqref{toric-kahler} with respect to the positive affine function
$f_\ah(\mc)=\ah_0 + \ah_1\mc_1 + \cdots + \ah_m\mc_m$ with $\ah_0\neq 0$ is a
toric K\"ahler metric of the form \eqref{toric-kahler} on $\tilde \Delta\op
\times \R^m$ with respect to momentum-angle coordinates $\tilde\mm,\tilde t$
and symplectic potential $\tilde u$ given by $\ip{\tilde\mm,a}=1$,
\begin{equation}\label{momentum-angle-twist}
\tilde\mm_j = \frac{\mm_j}{f_\ah(\mc)},\quad
\tilde t_j := t_j -\frac{\ah_j}{\ah_0}\, t_0, \quad j\in\{1,\ldots m\}\quad 
\text{and}\quad \tilde u(\tilde\mc) = \frac{u(\mc)}{f_\ah(\mc)}.
\end{equation}
\end{lemma}
As shown in~\cite{legendre1,legendre2}, when $M$ is compact, the rescaling
$\mc\mapsto\tilde\mc$ sends the polytope $\Delta$ in $\As$ to a polytope
$\tilde \Delta \sub \tilde \As$, and we have a compact CR $f$-twist $(\tilde
M,\tilde g,\tilde\omega)$ provided this polytope is rational.

\begin{ex}\label{ex:toric-RS} We illustrate the CR twist in the simple case
of a toric Riemann surface metric
\begin{equation}\label{eq:toric-RS}
g_V = \frac{\d \mm_1^2}{\Afn(\mm_1)}+\Afn(\mm_1)\, \d t_1^2,
\qquad \omega_V =\d\mm_1\wedge\d t_1,
\end{equation}
with \emph{profile function} $A$. This is a Sasaki--Reeb quotient of a Sasaki
$3$-manifold $(N,\Ds,J,\sas)$ with
\[
\cf_\sas = \d t_0+\mm_1\d t_1, \qquad J\d t\restr\Ds
= (J\d t_0\restr\Ds,J\d t_1\restr\Ds)
= -\frac{(-\mm_1,1)}{\Afn(\mm_1)}\,\d\mm_1\restr\Ds,
\]
using the affine chart $\mc=(1,\mc_1)$ as in Example~\ref{ex:tor3}. Thus
$\Afn(\mc_1)=1/u''(\mc_1)$ for a symplectic potential $u=u(\mc_1)$, and
there are straightforward integral formulae
\[
\cL(u)(\mc_1)=\int^{\mc_1} \frac{(-x,1)}{\Afn(x)} \d x, \qquad
u(\mc_1)=\ip{\cL(u),(1,\mc_1)}=\int^{\mc_1} \frac{\mc_1-x}{\Afn(x)} \d x,
\]
for $u$ and its differential lift (i.e., projective Legendre transformation)
$\cL(u)$. Any CR $f_\ah$-twist, with $f_\ah(\mc)=\ah_0+\ah_1\mc_1$, is then a
Sasaki--Reeb quotient of $\Sm$ by the Sasaki structure $\cae=f_\ah(\mm)\sas$
with contact form $\cf_\cae=\cf_\sas/f_\ah(\mm)$. If $\ah\neq 0$, then as in
Lemma~\ref{l:toric-twist} we may set $t_0=\ah_0\tilde t_0$ and $t_1=\tilde
t_1+\ah_1\tilde t_0$, so that $\cf_\cae=\d \tilde t_0+\tilde\mm_1 \d\tilde
t_1$, with $\tilde\mm_1=\mm_1/(\ah_0+\ah_1\mm_1)$, has Sasaki--Reeb field
$X_\cae=\partial/\partial\tilde t_0$. We then compute
\begin{gather*}
J\d \tilde t\restr\Ds
=(J\d \tilde t_0\restr\Ds,J\d\tilde t_1\restr\Ds)
=-\frac{(-\tilde\mm_1,1)}{\tilde\Afn(\tilde\mm_1)}\,\d\tilde\mm_1\restr\Ds\\
\tag*{with}
\tilde z_1= \frac{z_1}{a_0 + a_1 z_1} 
\quad\text{and}\quad
\tilde\Afn(\tilde\mc_1)=\frac{\ah_0^2\,\Afn(\mc_1)}{(\ah_0+\ah_1\mc_1)^3}.
\end{gather*}
However, other choices can be convenient: for example if the momentum image
$\mm(V)$ is $[-1,1]$ and $f_\ah\neq 0$ on $[-1,1]$ (i.e., $|\ah_0|>|\ah_1|$)
then we can instead preserve $[-1,1]$ with
\[
\tilde \mc_1 =\frac{\ah_0 \mc_1 +\ah_1}{\ah_0 +\ah_1 \mc_1} \quad\text{and}\quad
\tilde\Afn(\tilde\mc_1)=\frac{(\ah_0^2-\ah_1^2)^2}{(\ah_0+\ah_1\mc_1)^3}
\,\Afn(\mc_1).
\]
\end{ex}

\subsection{The generalized Calabi ansatz}\label{s:generalized-calabi}

We now discuss CR twists of K\"ahler metrics on certain toric fibre bundles
$\pi\colon M\to \Bm$ over a K\"ahler base manifold $(\Bm, g_\Bm, \omega_\Bm)$.
We follow the approach in~\cite{ACGT4}, to which we refer the reader for
further details, recalling here only the special case in which we are
interested.

Let $(V, g_V, \omega_V,\T^\ell)$ be a toric K\"ahler $2\ell$-manifold and let
$\pi\colon P \to \Bm$ be a principal $\T^\ell$-bundle over a K\"ahler
$2d$-manifold $(\Bm, g_\Bm, \omega_\Bm)$, equipped with a connection $1$-form
$\theta \in \Omega^1(P,\tor)$ ($\tor$ being the Lie algebra of $\T^\ell$) such
that
\begin{equation}\label{curvature}
\d\theta = \zeta\otimes \omega_\Bm, \quad \zeta \in \tor.
\end{equation}
As before, we assume $\tor=\torh/\spn{\eps}$, where $\torh$ may be identified
with the space of affine functions on the image $\Delta\sub \torh^*$ of the
momentum map of $V$.  Using~\eqref{eq:fund-seq}, we choose $\ah\in\torh$ such
that $\delta\ah=\zeta$ and the affine function $f_\ah(\mc):= \ip{\ah,\mc}$
is positive on $\Delta$.

Given these data, we can construct a K\"ahler $2m$-manifold $(M, g,
\omega)$, with $m=\ell+d$, $M= P\times _{\T^\ell} V \xrightarrow{\pi} \Bm$,
and
\begin{equation}\label{eq:gca}
\begin{split}
g  &=  f_\ah(\mm)\pi^*g_\Bm + \ip{\d \mm, G(\mm),  \d \mm}
+ \ip{\theta, H(\mm), \theta }, \\
\omega &=  f_\ah(\mm)\pi^*\omega_\Bm + \ip{ \d \mm \wedge \theta},\qquad
\d\theta = \delta\ah \otimes \omega_\Bm, 
\end{split}
\end{equation}
where $G$ and $H$ are determined by the toric K\"ahler metric on $V$ in
momentum--angle coordinates~\eqref{toric-kahler}. We refer to~\eqref{eq:gca}
as \emph{the generalized Calabi ansatz}, with data $(V, g_V, \omega_V)$,
$(\Bm, g_\Bm, \omega_\Bm)$ and $\ah\in\torh$. For any $\bh\in\torh$, the
affine function $f_\bh(\mc)=\ip{\bh,\mc}$ pulls back to Killing potentials for
both $(g_V, \omega_V)$ and $(g, \omega)$, and their CR $f_\bh$-twists are
related as follows.

\begin{prop}\label{p:generalized-calabi} Let $(M, g, \omega)$ be given by
the generalized Calabi ansatz for data $(V, g_V, \omega_V)$, $(\Bm, g_\Bm,
\omega_\Bm)$ and $\ah\in\torh$. Then for $\bh\in\torh$ with $f_\bh>0$ on $M$,
the generalized Calabi ansatz, with data a CR $f_\bh$-twist $(V_\bh, g_\bh,
\omega_\bh)$ of $V$, $(\Bm, g_\Bm, \omega_\Bm)$ and $\ah\in\torh$, is a CR
$f_\bh$-twist of $M$. In particular, the K\"ahler product of $(\Bm, g_\Bm,
\omega_\Bm)$ and $(V_\ah, g_\ah, \omega_\ah)$ is a CR $f_\ah$-twist of $M$.
\end{prop}
\begin{proof} It is enough to prove the result, for arbitrary $V$ and $\Bm$,
in the case $\ah=\eps$, with $(M,g,\omega)$ being the K\"ahler product of
$(\Bm, g_\Bm, \omega_\Bm)$ and $(V, g_V, \omega_V)$. Indeed, we may then
recover the K\"ahler metric~\eqref{eq:gca}, associated to a given $(\Bm, g_\Bm,
\omega_\Bm)$, $(V, g_V, \omega_V)$ and $\ah$, as a CR twist of the K\"ahler
product with $(\Bm, g_\Bm, \omega_\Bm)$ of a CR $f_\ah$-twist $(V_\ah, g_\ah,
\omega_\ah)$ of $(V, g_V, \omega_V)$, by taking $\bh=\ah$.

The Sasaki structure $(\Sm, \Ds, J, \sas=\cae_\eps)$ associated to the
K\"ahler product of $\Bm$ and $V$ is (locally) defined by the contact form
\begin{equation*}
\cf_\sas = \ip{\hat\mm,\d t+\theta_B\otimes\eps} =
\sum_{k=0}^m \hat\mm_j \d t_j  +\theta_\Bm 
\end{equation*}
where $\theta_\Bm$ is a (local) $1$-form on $\Bm$ with $\d\theta_\Bm=
\omega_\Bm$, and the second expression uses a basis of $\torh$ (for which we
may assume $e_0=\eps$ so that $\hat\mm_0 \equiv 1$ and
$X_\sas=\partial/\partial t_0$). The CR structure is determined from
$\d\cf_\sas = \pi^*(\omega_\Bm + \omega_V)$ and $g_\sas=\pi^*(g_\Bm +
g_V)\restr\Ds$.

Now let $f_\bh(\mc) = \sum_{k=0}^m \bh_k \mc_k$ be a positive affine function
defining new affine coordinates $\tilde \mc_j = \mc_j/f_\bh(\mc)$ on
$\Proj(\torh^*)$. The symplectic form $\tilde\omega$ on any Sasaki--Reeb
quotient $\tilde M$ of $\Sm$ by $\cae_\bh$ pulls back to
$\d(\cf_\sas/f_\bh(\hat\mm)) =\d\ip{\hat\mm,\d t+\theta_B\otimes\eps}$ and
hence is given by
\[
\tilde\omega = \ip{\tilde\mm,\d\theta} + \ip{\d\tilde\mm\wedge\theta} =
\ip{\tilde\mm,\eps}\omega_\Bm + \ip{\d\tilde\mm\wedge\theta},
\]
where $\tilde\mm$ is the pullback to $\tilde M$ of $\tilde\mc$ on
$\Proj(\torh^*)$ and $\theta\in\Omega^1(\tilde N,\tilde\tor)$, with
$\tilde\tor=\torh/\spn{\bh}$, pulls back to $\d t+\theta_B\otimes\eps$ mod
$\bh$ on $\Sm$. The complex structure on $B$ is unaffected by the CR twist,
while consideration of the action of the CR structure on $d\tilde
\mm\restr{\Ds}$ allows us to identify the toric fibres of $\tilde M$ over $B$
with the CR $f_\bh$-twist of $V$, as in Lemma~\ref{l:toric-twist}. As in that
Lemma, we can make the momentum--angle coordinates more explicit in a basis of
$\torh$ with $e_0=\eps$ and $\bh_0\neq 0$. In any case, the result now
follows.
\end{proof} 

\begin{cor}\label{cor:natural} If $(\Bm, g_\Bm, \omega_\Bm)$ is an extremal
K\"ahler manifold and $(V, g_V, \omega_V)$ is an $(f_\ah, \ell+2)$-extremal
K\"ahler manifold then $(g, \omega)$ given by \eqref{eq:gca} is $(f_\ah,
m+2)$-extremal.  In particular, a $\C P^\ell$-bundle $(M, J)=P(L_0 \oplus
\cdots \oplus L_\ell) \to \Bm$ over an extremal Hodge K\"ahler manifold $(\Bm,
g_\Bm, \omega_\Bm)$ has a natural $1$-parameter family of $(f_\ah,
m+2)$-extremal K\"ahler metrics.
\end{cor}
\begin{proof} The CR $f_\ah$-twist of the K\"ahler metric \eqref{eq:gca}
defined by Proposition~\ref{p:generalized-calabi} is the K\"ahler product of
$(\Bm, g_\Bm, \omega_\Bm)$ with the CR $f_\ah$-twist $(V_\ah, g_\ah,
\omega_\ah)$ of $(V, g_V, \omega_V)$. As $(V, g_V, \omega_V)$ is $(\ell+2,
f_\ah)$-extremal, $(V_\ah, g_\ah, \omega_\ah)$ is extremal by
Theorem~\ref{thm:main}. It follows that the CR $f_\ah$-twist of \eqref{eq:gca}
is extremal, so by Theorem~\ref{thm:main} again, we conclude that
\eqref{eq:gca} is $(f_\ah, m+2)$-extremal.

In the special case $M = P (L_0 \oplus \cdots \oplus L_\ell) \to \Bm$, we
can apply the above construction with $(V, g_V, \omega_V)=(\C P^\ell,
g_{FS}, \omega_{FS}),$ where $(g_{FS}, \omega_{FS})$ is a Fubini--Study metric
on $\C P^\ell$. By Corollary~\ref{c:FS}, $g_{FS}$ is $(f_\ah,
\ell+2)$-extremal, so the claim follows.
\end{proof}

\subsection{The Calabi ansatz}\label{s:Calabi}

We now specialize to the case that $(V,g_V,\omega_V)$ in the generalized
Calabi ansatz is a toric $2$-manifold or orbifold~\eqref{eq:toric-RS}; this is
the original Calabi ansatz when $V=\C P^1$, and~\eqref{eq:gca} reduces to
\begin{equation}\label{ruled-kahler}
g = (\ah_0+\ah_1\mm_1)g_\Bm + \frac{d\mm_1^2}{\Afn(\mu_1)} + \Afn(\mm_1)\theta^2,
\qquad \omega = (\ah_0+\ah_1\mm_1) \omega_\Bm + d\mm_1 \wedge \theta,
\end{equation}
where $d\theta = \ah_1 \omega_\Bm$. By Proposition~\ref{p:generalized-calabi},
$(M,g,\omega)$ has a CR $f_\ah$-twist $(\tilde M,\tilde g,\tilde \omega)$
given by the the K\"ahler product of $(\Bm,g_\Bm,\omega_\Bm)$ and $(\tilde
V,g_{\tilde V},\omega_{\tilde V})$, where the latter is a CR $f_\ah$-twist of
$(V,g_V,\omega_V)$ as in Example~\ref{ex:toric-RS}. Furthermore, by
Theorem~\ref{thm:main}, for any $\bh\in\torh$, $g$ is $(f_\bh,m+2)$-extremal
if and only if $\tilde g$ is $(\tilde f_\bh,m+2)$-extremal, where
$f_\bh(\mm_1)$ and $\tilde f_\bh(\tilde \mm_1)$ are the Killing potentials
induced by $\bh$ on $M$ and $\tilde M$ respectively.

If $\ah$ and $\bh$ are linearly independent then $\tilde f_\bh$ is nonconstant
and, up to homothety, we may assume that $\tilde
f_\bh(\tilde\mc_1)=\tilde\mc_1+\tilde\bh_0$. Then $\tilde g$ is $(\tilde
f_\bh,m+2)$-extremal iff $g_\Bm$ has constant scalar curvature $s_\Bm$ and
$g_{\tilde V}$ has profile function $\tilde A$ with
\begin{equation}\label{f-extremal-product}
\tilde\Afn(\tilde\mc_1-\tilde\bh_0)
= \ca_0\tilde\mc_1^{m+2} + \ca_1\tilde\mc_1^{m+1} - s_\Bm\tilde\mc_1^m
+ \ca_3\tilde\mc_1 + \ca_4.
\end{equation}
If $(\Bm, g_\Bm, \omega_\Bm)$ is a CSC Hodge K\"ahler manifold (where we may
assume without loss that $[\omega_\Bm/2\pi]$ is primitive) then this picture
globalizes in a couple of ways as follows.

\smallbreak

First, we may start from a weighted projective line $\tilde V=\C P^1_\wv$,
where $\wv=(w_-, w_+)$ is a pair of positive integers.  We equip $\C P^1_\wv$
with the toric symplectic structure $\omega_\wv$ induced by the quasi-regular
Sasaki structure $(\Sph^3, \Ds, J, \sas_\wv)$ on the $3$-sphere $\Sph^3 \sub
\C^2$.  By \eqref{eq:fund-seq}, the rational Delzant
polytope~\cite{Delzant,LT} of $(\C P^1_{\wv}, \omega_\wv)$ in the affine chart
defined by $\sas_\wv$ is given by $\{(\mc_0, \mc_1) : \mc_i \geq 0,\; w_-\mc_0
+ w_+\mc_1 =1 \}$, but we use instead the parametrization $\mc_0=
(1+\tilde\mc_1)/(2w_+)$, $\mc_1 = (1-\tilde\mc_1)/(2w_-)$ to realize this
rational Delzant polytope as the interval $[-1, 1]$ with inward normals
$1/(2w_+)$ and $-1/(2w_-)$. As explained in \cite{BG-book}, for any positive
integer $k$, the product K\"ahler orbifold $(\Bm, g_\Bm, \omega_\Bm) \times
(\C P^1_\wv, k g_\wv, k \omega_\wv)$ gives rise to a compact quasi-regular
extremal Sasaki orbifold $(\Sm_{\wv,k}, \Ds, J, \sas_{\wv,k})$, which is the
Sasaki join of the regular Sasaki manifold $(\Sm_\Bm, \Ds_\Bm, J_\Bm,
\sas_\Bm)$ associated to $(\Bm, g_\Bm, \omega_\Bm)$ and $(\Sph^3, \Ds, J,
\frac{1}{k}\sas_\wv)$. There are well-understood conditions in terms of the
integers $(w_+, w_-,k)$ ensuring that $\Sm_{\wv,k}$ is a smooth manifold, see
\cite{BG-book}. Now any $\tilde\Afn$ given by \eqref{f-extremal-product} with
$|\tilde\bh_0|>1$, which satisfies the well-known positivity and boundary
conditions
\begin{equation}\label{compactification}
  \tilde\Afn(\tilde\mc_1)>0 \quad \mathrm{on} \quad (-1, 1), \quad
  \tilde\Afn(\pm 1)=0 \quad \textrm{and}\quad \tilde\Afn'(\pm 1)= \mp 4
  w_{\mp}/k,
\end{equation}
gives rise to a toric, $k\omega_\wv$-compatible K\"ahler metric $\tilde g_\wv$
on $\C P^1_\wv$, such that the product metric $g_\Bm + \tilde g_\wv$ is
$(\tilde f_\bh, m+2)$-extremal.  We thus get a new Sasaki structure
$(\Sm_{\wv,k}, \Ds, J_\wv, \cae_\bh)$ which is extremal by
Theorem~\ref{thm:main}. Note that $\cae_\bh$ is not quasi-regular if
$\tilde\bh_0$ is irrational.

For a fixed $\ah_0$, the endpoint conditions \eqref{compactification}
determine the unknown coefficients $\ca_0$, $\ca_1$, $\ca_3$ and $\ca_4$ of a
polynomial $\tilde\Afn$ satisfying \eqref{f-extremal-product}, and it remains
to examine the positivity condition for $\tilde\Afn$.  This is therefore an
effective tool for generating compact examples of extremal Sasaki metrics,
providing an explanatory framework for the constructions in \cite{BT0, BT}.

\smallbreak

Secondly, we may begin instead with $M=\Proj(\cO\oplus\cL)$ where $\cL$ is a
holomorphic line bundle over $\Bm$ such that $c_1(\cL)=\ell[\omega_\Bm/2\pi]$
for $\ell\in\Z^+$ (and $\cO$ denotes the trivial line bundle).
Then~\eqref{ruled-kahler} defines a K\"ahler metric on $M$ such that the
$\Sph^1$-action induced by scalar multiplication in $\cO$ is isometric and
hamiltonian with momentum map $\mm_1$ and momentum image $\mm_1(M)=[-1, 1]
\subset \R$ if and only if $\Afn(\mc_1)$ is a smooth function on $[-1,1]$
satisfying the boundary conditions
\begin{equation}\label{boundary}
\Afn(\pm 1)=0, \quad\Afn'(\pm 1)= \mp 2,
\end{equation}
and the positivity condition
\begin{equation}\label{positive}
\Afn(\mc_1)>0 \quad \textrm{on} \quad (-1,1),
\end{equation}
and $\ah_0+\ah_1\mc_1$ is positive on $[-1,1]$ with $|\ah_1|=\ell$ (and we may
assume $\ah_1=\ell$ by replacing $\mc_1$ with $-\mc_1$ if necessary). Here
$\theta$ is the connection form associated to a principal $\Sph^1$-connection
on the unit circle bundle in $M\to\Bm$ and
\begin{equation}\label{kahler-class}
[\omega/2\pi] = c_1(\cO_{\cO\oplus\cL}(2) )+ (\ah_0+\ah_1)c_1(\pi^* \cL).
\end{equation}
For any positive integers $\p, \q$ such that $\q/\p > \ell$, $L_{\p,\q} :=
\cO_{\cO \oplus \cL}(\p) \otimes \pi^* \cL^{\q/\ell}$ is a polarization on
$M$, $c_1(L_{\p,\q})$ being homothetic to a K\"ahler class of the form
\eqref{kahler-class} with $\ah_0= (2\q/\p) -\ell$ and $\ah_1=\ell$. We thus
let $(\Sm_{\p,\q}, \Ds, J, \sas)$ be the smooth Sasaki manifold
corresponding to the K\"ahler manifold $(M, \frac{\p}{2} g,
\frac{\p}{2}\omega)$ via Example~\ref{e:regular}, where $(g, \omega)$ is given
by \eqref{ruled-kahler} (with $\ah_0= (2\q/\p) -\ell$ and $\ah_1=\ell$). Up to
a covering, $(\Sm_{\p,\q}, \Ds, \sas)$ is determined by the ratio $\q/\p$, so
we assume henceforth that $\p$ and $\q$ are coprime positive integers.  In
\cite[(37)]{BT}, the contact manifold $(\Sm_{\p,\q}, \Ds)$ is identified with
the Sasaki join $(\Sm_{\wv,\p}, \Ds)$ constructed over $B\times \C P^1_\wv$
above, with weights $w_+ = \q, w_- = \q-\p\ell$. The theory of CR twists
further identifies the CR structure $J$ on $(\Sm_{\p,\q}, \Ds)$ induced
by~\eqref{ruled-kahler} with the CR structure $J_\wv$ on $(\Sm_{\wv,\p}, \Ds)$
induced by
\begin{align*}
\tilde g &= g_\Bm + \frac{\d\tilde\mm_1^2}{\tilde\Afn(\tilde\mm_1)}
+\tilde\Afn(\tilde\mm_1)\d t^2,
\qquad\tilde\omega = \omega_\Bm+  \d\tilde\mm_1 \wedge \d t, \\
\tag*{where}
\tilde\Afn(\tilde\mc_1)&=
\frac{(\ah_0^2-\ah_1^2)^2\Afn(\mc_1)}{(\ah_0+\ah_1\mc_1)^3},
\quad \mc_1 =\frac{\ah_0\tilde\mc_1-\ah_1}{\ah_0-\ah_1\tilde\mc_1},\quad
\ah_0=\frac{2\q}{\p} -\ell \quad
\text{and} \quad \ah_1=\ell.
\end{align*}

\section{Separable toric geometries}\label{s:ansatz}

\subsection{Regular ambitoric structures}

In~\cite{ambitoric1,ambitoric2}, the following $4$-dimensional geometric
structure was studied.
\begin{defn} An \emph{ambik\"ahler structure} on a real $4$-manifold or
orbifold $M$ consists of a pair of K\"ahler metrics $(g_-, J_-, \omega_-)$ and
$(g_+, J_+, \omega_+)$ such that
\begin{bulletlist}
\item $g_-$ and $g_+$ are conformally equivalent;
\item $J_-$ and $J_+$ have opposite orientations.
\end{bulletlist}
The structure is said to be \emph{ambitoric} if in addition there is a
$2$-dimensional subspace $\tor$ of vector fields on $M$, linearly independent
on a dense open set, whose elements are hamiltonian and Poisson-commuting
Killing vector fields with respect to both $(g_-,\omega_-)$ and
$(g_+,\omega_+)$---i.e., both K\"ahler structures are locally toric.
\end{defn}
It was shown in~\cite{ambitoric1} that any ambitoric structure is locally
either a product, of Calabi type, or a \emph{regular} ambitoric structure
given by the following ansatz. Let $q(x)=q_0+2q_1 x+ q_2x^2$ be a quadratic
polynomial and let $M$ be a $4$-manifold or orbifold with real-valued
functions $(x_1,x_2,\tau_0,\tau_1,\tau_2)$ such that $x_1>x_2$,
$2q_1\tau_1=q_0\tau_2+q_2\tau_0$, and their exterior derivatives span each
cotangent space. Let $\tor$ be the $2$-dimensional space of vector fields $K$
on $M$ with $\d x_1(K)=0=\d x_2(K)$ and $\d\tau_j(K)$ constant, and let
$A(x)$ and $B(x)$ be positive functions on open neighbourhoods of the
images of $x_1$ and $x_2$ in $\R$, on whose product
$f_q(x_1,x_2):=q_0+q_1(x_1+x_2)+q_2x_1x_2$ is positive.  Then $M$ is ambitoric
with
\begin{equation}\label{eq:ambitoric}
\begin{split}
g_\pm &= \biggl(\frac{x_1-x_2}{f_q(x_1,x_2)}\biggr)^{\!\!\pm1}
\biggl(\frac{\d x_1^2}{A(x_1)} + \frac{\d x_2^2}{B(x_2)}+A(x_1)
\alpha_1^{\,2}+ B(x_2)\alpha_2^{\,2}\biggr),\\
\omega_\pm &= \biggl(\frac{x_1-x_2}{f_q(x_1,x_2)}\biggr)^{\!\!\pm 1}
(\d x_1\wedge \alpha_1\pm \d x_2 \wedge \alpha_2),\qquad
\begin{aligned}
J_\pm \d x_1 &= A(x_1) \alpha_1,\\
J_\pm \d x_2 &= \pm B(x_2)\alpha_2,
\end{aligned}\\
\alpha_1 &=\frac{\d\tau_0 + 2x_2 \d\tau_1 + x_2^2 \d\tau_2}
{(x_1-x_2)f_q(x_1,x_2)},\qquad
\alpha_2 =\frac{\d\tau_0 + 2x_1 \d\tau_1 + x_1^2 \d\tau_2}
{(x_1-x_2)f_q(x_1,x_2)}
\end{split}
\end{equation}
There is a gauge freedom to make a simultaneous projective transformation of
the coordinates $x_1$, $x_2$, with $q$ transforming as a quadratic polynomial,
and $A,B$ as quartics~\cite{ambitoric1}. If $q$ has repeated roots,
we may use this freedom to set $q=1$, and then $g_+$ is a $2$-dimensional
\emph{orthotoric metric}, as studied in~\cite{wsdk,HFKG1}. We then refer to
$g_-$ as a \emph{negative orthotoric metric}.

Ambitoric structures are examples of \emph{separable toric geometries}, i.e.,
they admit \emph{separable coordinates} $x_1,\ldots x_m$ in which the metric
is determined by $m$ functions of $1$ variable (and some explicit data, such
as $q$ here). We now explore CR twists for some separable toric geometries.
While we could simply apply the general approach given in
Lemma~\ref{l:toric-twist}, this is not expedient for a couple of reasons. On a
practical level, we would need to compute: the transformation from separable
coordinates to momenta, the symplectic potential, its CR twist in terms of the
new momenta, and finally the transformation from these momenta back to
separable coordinates. This is rather involved, and unnecessarily so, because
whereas a CR twist involves a change of momentum coordinates due to the change
of affine chart, the separable coordinates remain fixed. We illustrate this
first in the simplest separable situation: K\"ahler products of toric Riemann
surfaces.

\subsection{The CR twisted toric product ansatz}\label{s:twist-product}

A K\"ahler product of toric Riemann surfaces has a K\"ahler metric of the form:
\begin{equation}\label{eq:tor-prod}
g = \sum_{i=1}^m\Bigl(\frac{{\d x_i}^2}{\Afn_i(x_i)} + \Afn_i(x_i){\d t_i}^2\Bigr),
\quad\omega=\sum_{i=1}^m \d x_i\wedge \d t_i, \quad J\d x_i = \Afn_i(x_i)\d t_i,
\end{equation}
where $\Afn_1,\ldots \Afn_m$ are arbitrary functions of $1$ variable. In this case
the separable coordinates and momenta coincide: the toric Killing potentials
have the form
\[
f_\bh(x_1,\ldots x_m) = \bh_0+\bh_1 x_1+\cdots +\bh_m x_m.
\]
It is straightforward to compute the CR structure associated to $(g,\omega,J)$
as in Example~\ref{e:regular}. Denoting by $t_i, x_i$ also their pullbacks to
$\Sm$, we have $\Ds=\ker\cf$ and $J\colon\Ds^*\to\Ds^*$ given by
\begin{equation*}
\cf = \d t_0+\sum_{i=1}^m x_i \d t_i, \qquad J\bigl(\d x_i\restr\Ds\bigr)
= \Afn_i(x_i)\d t_i\restr\Ds.
\end{equation*}
We now lift $f_\bh(x_1,\ldots x_m)$ to a new Sasaki structure $\cae_\bh$ on
$\Sm$ and compute the new Sasaki--Reeb quotient. The new contact form is
$\cf_\bh:=\cf_{\cae_\bh}=\cf/f_\bh$, with
\[
\d\cf_\bh= \sum_{i=1}^m\d x_i\wedge\frac{\partial \cf_\bh}{\partial x_i}
= \frac1{f_\bh(x_1,\ldots x_m)} \sum_{i=1}^m \d x_i\wedge (\d t_i -\bh_i\cf_\bh).
\]
Since $J\bigl(\d x_i\restr\Ds\bigr) = \Afn_i(x_i)(\d t_i
-\bh_i\cf_\bh)\restr\Ds$, the Sasaki--Reeb quotient is given by the following
toric ansatz, originally proposed in~\cite{ACGL}, which we refer to here as a
\emph{twisted toric product}:
\begin{equation}\label{eq:tw-tor-prod}
\begin{gathered}
g_\bh = \frac{1}{f_\bh(x_1,\ldots x_m)}
\sum_{i=1}^m\Bigl(\frac{{\d x_i}^2}{\Afn_i(x_i)} + \Afn_i(x_i){\alpha_i}^2\Bigr),
\qquad \omega_\bh = \frac{1}{f_\bh(x_1,\ldots x_m)} \sum_{i=1}^m
\d x_i \wedge \alpha_i,\\
J_\bh\d x_i =\alpha_i, \qquad\d\alpha_i= - \bh_i \omega, \qquad
f_\bh(x_1,\ldots x_m) = \bh_0+\bh_1x_1 + \cdots + \bh_mx_m.
\end{gathered}
\end{equation}
For $\bh_0\neq0$, we may obtain more explicit angle coordinates by setting
$\tau_i= t_i-\bh_it_0/\bh_0$ and $\alpha_i= \d\tau_i - (\bh_i/f_\bh)
\sum_{j=1}^m x_j \d\tau_j$. We may take as momenta
\[
\mm_0 = \frac1{f_\bh(x_1,\ldots x_m)},\qquad
\mm_j = \frac{x_j}{f_\bh(x_1,\ldots x_m)}, \qquad j\in\{1,\ldots m\}.
\]
Hence the momentum coordinates and separable coordinates no longer agree.  The
original product metric~\eqref{eq:tor-prod} is a CR twist
of~\eqref{eq:tw-tor-prod} by $\mm_0=f_\bh^{-1}=1/f_\bh$. It was shown in
\cite{ACGL} that when $m=2$, this construction unifies the ambitoric product,
Calabi and negative orthotoric ansatz of \cite{ambitoric1} in a single family.

It was also shown in \cite{ACGL} that a twisted toric product
metric~\eqref{eq:tw-tor-prod} is $(f_\bh^{-1},m+2)$-extremal if and only if
$\Afn_j$ is a cubic polynomial for all $j\in\{1,\ldots m\}$. We see here that
this follows straightforwardly from Theorem~\ref{thm:main},
as~\eqref{eq:tw-tor-prod} is $(f_\bh^{-1},m+2)$-extremal if and only
if~\eqref{eq:tor-prod} is extremal, and a toric K\"ahler product is extremal
if and only if the factors are, meaning that each $\Afn_i$ is a cubic. In this
case, we may also identify the CR manifold which has these metrics as its
Sasaki--Reeb quotients. Indeed, straightforward computation shows that the
Cartan tensor of a Sasaki $3$-manifold vanishes precisely when the
(transversal, i.e., Tanaka--Webster) scalar curvature is transversally
holomorphic (see e.g.~\cite{herzlich}). It then follows from \cite{Car} that
the $(f_\bh^{-1}, m+2)$-extremal metrics given by \eqref{eq:tw-tor-prod} are
obtained as Sasaki--Reeb quotients with respect to the CR structure of a
(local) Sasaki join~\cite{BG-book} of $m$ copies of the standard CR structure
$(\Ds_0, J_0)$ on the $3$-sphere $\Sph^3 \sub \C^2$, with respect to a (local)
Sasaki structure on each factor.

We next consider the extremality condition for the K\"ahler metrics
\eqref{eq:tw-tor-prod}, using again Theorem~\ref{thm:main} to infer that
\eqref{eq:tw-tor-prod} is extremal if and only if the product
metric~\eqref{eq:tor-prod} is $(f_\bh,m+2)$-extremal. We thus have
\begin{multline}\label{extremal-tw-prod}
\sum_{j=1}^m \biggl( -f_\bh^2 \Afn_j''(x_j) + 2(m+1)
f_\bh\frac{\partial f_\bh}{\partial x_j} \Afn_j'(x_j)-(m+1)(m+2)
\biggl(\frac{\partial f_\bh}{\partial x_j}\biggr)^2 \Afn_j(x_j)\biggr)\\
=-\sum_{j=1}^m f_\bh^{m+3}\frac{\partial^2}{\partial
  x_j^2}\biggl(\frac{\Afn_j(x_j)}{f_\bh^{m+1}}\biggr)
= \Scal_{f_\bh,m+2}(g) = c_0 + c_1x_1 + \cdots + c_mx_m,
\end{multline}
where $c_0, c_1, \ldots c_m$ are some real constants. For $m=1$ we get that
$\Afn_1$ must be a polynomial of degree $\leq 3$ and for $m=2$
\eqref{eq:tw-tor-prod} is given by the ambitoric product, Calabi or negative
orthotoric ansatz of \cite{ambitoric1}, and the extremality condition
\eqref{extremal-tw-prod} can be solved~\cite{ambitoric1} in terms of two
polynomials $\Afn_1$ and $\Afn_2$ of degree $\leq 4$. We thus assume from now on
that $m\geq 3$.

\begin{prop} For $m\geq 3$ the K\"ahler metric \eqref{eq:tw-tor-prod} is
extremal if and only if it is a product of extremal Riemann surfaces, or is
given by the Calabi ansatz over a product of $m-1$ CSC Riemann surfaces, or is
the K\"ahler product of a scalar-flat product of Riemann surfaces with a
product of flat Riemann surfaces, as in Proposition~\textup{\ref{flat-local}}.
\end{prop}
\begin{proof}  Differentiating \eqref{extremal-tw-prod} $(m+1)$ times with
respect to $x_j$ yields
\begin{equation*}
f_\bh^{2} \Afn_j^{(m+3)}(x_j) =0,
\end{equation*}
showing that each $\Afn_j$ must be a polynomial of degree $\leq m+2$. Thus,
both sides of \eqref{extremal-tw-prod} are polynomials in $x_i$, so we may
compare coefficients. Taking two derivatives in $x_j$ gives 
\begin{equation*}
0 =  f_\bh^{m} \frac{\partial^2}{\partial x_j^2}
\biggl(\frac{\Afn_j^{(2)}(x_j)}{f_\bh^{m-1}}\biggr) 
=f_\bh^2 \Afn_j^{(4)}(x_j)-2(m-1)\bh_j f_\bh \Afn_j^{(3)}(x_j)
+m(m-1) \bh_j^2\Afn_j^{(2)}(x_j).
\end{equation*}
If $\bh_i \neq 0$ for some $i\neq j$, the vanishing of the polynomial
coefficients containing $x_i^2$ in the above relation show that $\Afn_j$ has
degree $\leq 3$; if furthermore $\bh_j \neq 0$, then the coefficients
containing $x_i$ show that $\Afn_j$ has degree $\leq 2$. Substituting back in
\eqref{extremal-tw-prod} and comparing coefficients, this yields the following
three possibilities for the solutions to \eqref{extremal-tw-prod} with $m\geq
3$:
\begin{bulletlist}
\item $f_\bh(x_1,\ldots x_m)=\bh_0$. Then the $\Afn_j$ are polynomials of degree
  $\leq 3$ and the corresponding extremal metric \eqref{eq:tw-tor-prod} is a
  product of extremal Riemann surfaces;
\item $f_\bh(x_1,\ldots x_m)= \bh_0 + \bh_j x_j$ with $\bh_j\neq 0$. Then
  \eqref{eq:tw-tor-prod} is given by the Calabi ansatz over the product of
  $(m-1)$ Riemann surfaces indexed by $i:i\neq j$. In particular, for each
  $i:i\neq j$, $\Afn_i$ is a polynomial of degree $\leq 2$ whereas $\Afn_j$ is
  a polynomial of degree $\leq m+2$, as described by
  \eqref{f-extremal-product}.
\item there are $0\neq j_1\neq j_j\neq 0$ with $\bh_{j_1}\neq 0\neq
  \bh_{j_2}$. Then for each $j$ with $\bh_j\neq 0$, $\Afn_j$ is a polynomial of
  degree $\leq 1$, and for each $i$ with $\bh_i=0$, $\Afn_i$ is a polynomial of
  degree $\leq 2$ with $\sum_{i:\bh_i=0} A''_i =0$. Thus, in this case,
  $g$ is the product metric of a scalar-flat product of (CSC) Riemann surfaces
  (indexed by $\{i: \bh_i=0\}$) with a product of flat Riemann surfaces
  (indexed by $\{j: \bh_j\neq 0\}$). \qedhere
\end{bulletlist}
\end{proof}
 
\subsection{CR twists of positive regular ambitoric structures}

We return now to regular ambitoric structures~\eqref{eq:ambitoric}, for which
it was shown in~\cite{ambitoric1} that $g_+$ is extremal if and only if $g_-$
is extremal if and only if
\begin{equation}\label{eq:ambi-extremal}
A=pq+P,\qquad B=pq-P,\\
\end{equation}
where $p$ is a quadratic polynomial orthogonal to $q$, and $P$ is polynomial
of degree $\leq 4$. Note that the orthogonality condition $\ip{p,q} :=p_0q_2 -p_1q_1 + p_2q_0=0$ means
that the roots of $p$ and the roots of $q$ harmonically separate each other.

When $q$ has distinct roots, the positive and negative structures are
equivalent, cf.~\cite[Remark~5]{ambitoric2}, while in the case of repeated
roots the negative orthotoric structures are twisted toric
products~\cite{ACGL}, as noted above. Hence we only need to consider the positive
ambitoric metrics. The CR structure associated to $(g_+,J_+,\omega_+)$
in~\eqref{eq:ambitoric} was computed in~\cite[App.~C]{ambitoric2}, which
implies that $\Ds=\ker\cf$ and $J\colon\Ds^*\to\Ds^*$ are given by
\begin{equation}\label{eq:ambi-CR}
  \cf= \frac{\d t_0 + (x_1+x_2) \d t_1 + x_1 x_2 \,\d t_2}{x_1-x_2},\qquad
\begin{aligned}
J(\d x_1\restr\Ds) &= A(x_1)
\frac{\d t_0 + 2x_2 \d t_1 + x_2^2 \d t_2}{(x_1-x_2)^2}\Big|_\Ds,\\
J(\d x_2\restr\Ds) &= B(x_2)
\frac{\d t_0 + 2x_1 \d t_1 + x_1^2 \d t_2}{(x_1-x_2)^2}\Big|_\Ds,
\end{aligned}
\end{equation}
independently of $q$, while the toric Killing potentials of
$(g_+,J_+,\omega_+)$ have the form $f_w/f_q$, where $w$ is quadratic
polynomial, and lift to Sasaki structures
$\cae_w:=f_w(x_1,x_2)\sas/(x_1-x_2)$, where $\sas\in\con_+(\Sm,\Ds)$ with
$\cf=\sas^{-1}\cf_\Ds$.

The CR structures arising from an extremal positive ambitoric metric thus have
$A$ and $B$ of degree $\leq 4$, such that the roots of
$A+B$ have harmonic cross-ratio (i.e., in $\{-1,1/2,2\}$). We may
then write
\begin{equation}\label{eq:new-ext}
A=p_1p_2+P, \quad B=p_1p_2-P,\quad\text{with}\quad
\deg P\leq 4, \quad \deg p_j\leq 2,\quad \ip{p_1,p_2}=0.
\end{equation}
Here we have renamed the quadratics compared to~\eqref{eq:ambi-extremal} so
that we are free to use $q$ to define an arbitrary Sasaki--Reeb quotient
of~\eqref{eq:ambi-CR}. Indeed for any quadratic $q$, the Sasaki structure
$\cae_q$ is $(f_{p_j}\sas,4)$-extremal for $j\in\{1,2\}$ by
Theorem~\ref{thm:main} and is extremal if $q=p_1$ or $q=p_2$. We obtain in
particular a result of~\cite{AM}, as any Sasaki--Reeb quotient by $\cae_q$,
given explicitly by~\eqref{eq:ambitoric}, subject to~\eqref{eq:new-ext}, is
$(f_{p_j}/f_q,4)$-extremal for $j\in\{1,2\}$, i.e., the scalar curvature of
$\tilde g_j=(f_q/f_{p_j})^2 g_+$ is a Killing potential of $g_+$; in fact one
can compute~\cite{ambitoric1,AM}
\begin{equation}
\Scal(\tilde g_j) = -\frac{f_w}{f_q}\quad
\text{with}\quad w:=\{p_j, (p_j, P)^{(2)}\},
\end{equation}
where the Poisson bracket is given by $\{p,r\}:= p'\,r - p\, r'$
and
\[
(p, P)^{(2)}:=p\,P''-3p'\,P+6p''\,P
\]
is a transvectant of $p$ and $P$. Special choices of $q$ give special metrics
in this family of Sasaki--Reeb quotients~\cite{ambitoric1,AM}.
\begin{bulletlist}
\item If $q=p_j$, then $\tilde g_j=g_+$, recovering the case that $g_+$ is
  extremal.
\item If $\ip{q,p_j}=0$, then $g_-$ is also $(f_{p_j}/f_q,4)$-extremal, and
  $\tilde g_j$ has diagonal Ricci tensor; if in addition $\ip{(p_j,
  P)^{(2)},q}=0$ then $\{p_j, (p_j, P)^{(2)}\}$ is a multiple of $q$; hence
  $\tilde g_j$ is CSC, so $g_+$ is conformally Einstein--Maxwell (in fact to a
  riemannian Pleba\'nski--Demia\'nski metric~\cite{ambitoric1,DKM,PD}).
\item Combining these observations, if say $q=p_1$, then $\ip{q,p_2}=0$ so
  $g_+=\tilde g_1$ is extremal, while $\tilde g_2$ has diagonal Ricci tensor;
  if in addition $\ip{(p_2, P)^{(2)},q}=0$ then $\tilde g_2$ is Einstein.
\item Finally, taking $q=1$, we obtain an orthotoric metric in the family.
\end{bulletlist}

\begin{cor}\label{ambitoric-vs-orthotoric} Any regular positive ambitoric
K\"ahler metric $(g_+, \omega_+, J_+)$ given by \eqref{eq:ambitoric} can be
obtained as a $f_q$-twist of an orthotoric metric.
\end{cor}

\subsection{The CR twisted orthotoric ansatz}

Corollary~\ref{ambitoric-vs-orthotoric} immediately suggests a higher
dimensional extension of the positive regular ambitoric
ansatz~\eqref{eq:ambitoric}. For this, we start with an orthotoric
$2m$-manifold $M$ with K\"ahler structure~\cite{HFKG1}:
\begin{equation} \label{orthotoric}\begin{split}
g &= \sum_{j=1}^m \biggl(\frac{\Delta_j} {\Afn_j (x_j)}\,\d x_j^2
+ \frac{\Afn_j (x_j)}{\Delta_j}
\Bigl (\sum_{r=1}^m \sigma_{r-1} (\hat x_j) \, \d t_r \Bigr )^2\biggr),\\
\omega &= \sum_{j=1}^m \d x_j \wedge\Bigl (
\sum_{r=1}^m \sigma_{r-1} (\hat x_j) \d t_r \Bigr )
= \sum_{r=1}^m \d \mm_r \wedge \d t_r,\\
J \d x_j &= \frac{\Afn_j (x_j)}{\Delta_j} \,
\sum_{r=1}^m \sigma_{r-1} (\hat x_j) \, \d t_r, \qquad\qquad
J \d t_r = (-1) ^r \,\sum_{j=1}^m \frac{x_j ^{m-r}}{\Afn_j(x_j)}\,\d x_j,
\end{split}\end{equation}
where each $\Afn_j$ is a smooth function of $1$ variable,
$\mm_r=\sigma_r(x_1,\ldots x_m)$ are the momentum coordinates ($\sigma_r$
being the $r$-th elementary symmetric function with $\sigma_0=1$) $\hat
x_j=(x_k:k\neq j)$, and $\Delta_j = \prod_{k\neq j}(x_j - x_k)$.  The
separable coordinates $\bx:=(x_1, \ldots x_m)$ are called \emph{orthotoric}
and have a natural gauge freedom under simultaneous affine changes $\tilde x_j
= ax_j +b$.  The toric Killing potentials in orthotoric coordinates are
\begin{equation}\label{polarization}
f_q(\bx)=q_0 + q_1\mm_1+\cdots +q_m \mm_m \quad\text{with}\quad
\mm_r=\sigma_r(\bx),
\end{equation}
which can be viewed as the polarized form of a degree $\leq m$ polynomial 
\begin{equation}\label{unpolarization}
q(x) := f_q(x, \ldots x)= \sum_{j=0}^m \binom{j}{m}q_j x^j.
\end{equation}
As usual, to write down the CR structure $(\Sm,\Ds,J)$ over $M$, it is
convenient to view $\d t=(\d t_0,\d t_1,\ldots \d t_m)$ as a $1$-form with
values in the Lie algebra $\torh\cong\R^{m+1}$ of Killing potentials (with
basis $\sigma_0,\sigma_1,\ldots \sigma_m$). Then $\Ds=\ker\cf$ and
$J\colon\Ds^*\to\Ds^*$ are given by
\begin{equation}\label{eq:q-CR}
\cf = \d t(\bx) := \sum_{r=0}^m \mm_r \d t_r,\qquad
J(\d x_j\restr\Ds) =  \frac{\Afn_j (x_j)}{\Delta_j}\frac{\partial (\d t(\bx))}
{\partial x_j}\Big|_\Ds
\end{equation}
with $\mm_0=1$ (omitting pullbacks to $\Sm$).

\begin{prop}\label{prop:orthotoric-twist} Let $(g, \omega, J)$ be the
orthotoric K\"ahler metric~\eqref{orthotoric} and let $f_q$ be a positive
function of the form \eqref{polarization}. Then a CR $f_q$-twist of $(g,
\omega, J)$ has toric K\"ahler metric
\begin{equation} \label{orthotoric-twist}
\begin{split}
g_q &= \sum_{j = 1}^m \biggl(\frac{\Delta_j}{\Afn_j (x_j)f_q(\bx)} \, \d x_j^2
+ \frac{\Afn_j (x_j)f_q(\bx)}{\Delta_j} \Bigl (
\frac{\partial}{\partial x_j} \frac{\d t(\bx)}{f_q(\bx)} \Bigr )^2\biggr),\\
\omega_q &= \sum_{j=1}^m \d x_j \wedge\Bigl (\frac{\partial}{\partial x_j}
\frac{\d t(\bx)}{f_q(\bx)} \Bigr ), \qquad
J_q \d x_j = \frac{\Afn_j (x_j)f_q(\bx)}{\Delta_j}
\Bigl(\frac{\partial}{\partial x_j}
\frac{\d t(\bx)}{f_q(\bx)}
\Bigr).
\end{split}
\end{equation}
\end{prop}
\begin{proof} The CR $f_q$-twist is the Sasaki--Reeb quotient of~\eqref{eq:q-CR}
  by the Sasaki structure $\cae_q=f_q\sas$ with contact form $\cf/f_q(\bx)$
  (i.e., $\sas$ is the Sasaki structure with contact form
  $\cf=\sas^{-1}\cf_\Ds$), and
\begin{equation*}
\d\Bigl(\frac{\cf}{f_q(\bx)}\Bigr)
= \sum_{j=1}^m \d x_j \wedge \frac{\partial}{\partial x_j}
\frac{\d t(\bx)} {f_q(\bx)}.
\end{equation*}
Now we observe that
\begin{equation*}
J(\d x_j\restr\Ds) =  \frac{\Afn_j (x_j)}{\Delta_j}\frac{\partial (\d t(\bx))}
{\partial x_j}\Big|_\Ds
= \frac{\Afn_j (x_j)}{\Delta_j} f_q(\bx) \frac{\partial}{\partial x_j}
\frac{\d t(\bx)}{f_q(\bx)}\Big|_\Ds
\end{equation*}
and the $1$-forms $\d x_j$ and $\frac{\partial}{\partial x_j} \frac{\d
  t(\bx)}{f_q(\bx)}$ are basic with respect to $\cae_q$---in the latter case
because $\bigl(\frac{\d t(\bx)}{f_q(\bx)}\bigr)(X_{\cae_q}) = 1$. Hence the
transversal K\"ahler structure of $\cae_q$ is the pullback
of~\eqref{orthotoric-twist}.
\end{proof}

We now turn to the extremality condition of the K\"ahler metrics given by
\eqref{orthotoric-twist}. By Theorem~\ref{thm:main}, such a metric is extremal
iff the orthotoric metric \eqref{orthotoric} is $(f_q, m+2)$-extremal, a
condition studied in \cite[App.~A]{AMT}. Using standard formulae for the
scalar curvature and laplacian of an orthotoric metric~\cite{HFKG1},
this condition is
\begin{multline}\label{f-ext-orthotoric}
\sum_{j=1}^m \biggl( -f_q(\bx)^2 \frac{\Afn_j''(x_j)}{\Delta_j} + 2(m+1)
f_q(\bx)\frac{\partial f_q}{\partial x_j} \frac{\Afn_j'(x_j)}{\Delta_j}
-(m+1)(m+2)\Bigl(
\frac{\partial f_q}{\partial x_j}\Bigr)^2
\frac{\Afn_j(x_j)}{\Delta_j}\biggr)\\
=-\sum_{j=1}^m \frac{f_q(\bx)^{m+3}}{\Delta_j}\frac{\partial^2}{\partial
  x_j^2}\biggl(\frac{\Afn_j(x_j)}{f_q(\bx)^{m+1}}\biggr)= \Scal_{f_q,m+2}(g) =
\sum_{k=0}^m c_k\mm_k.
\end{multline}
for some constants $c_0, \ldots c_m$.  If we let $\Afn_j(x)=P(x)$ for a
$j$-independent polynomial $P$ of degree $\leq m+2$, the metric
\eqref{orthotoric} is Bochner-flat by \cite{HFKG1}, and we obtain (for any
$q$) a solution of \eqref{f-ext-orthotoric} by Proposition~\ref{BF-local}.
When $q=1$, \eqref{f-ext-orthotoric} describes the extremality condition for
\eqref{orthotoric}, which is studied in \cite{HFKG1}, where the solutions are
given as $\Afn_j(x) = P(x) + \ca_{j1} x + \ca_{j0}$ for a $j$-independent
polynomial $P$ of degree $\leq m+2$ and arbitrary real constants $\ca_{j1},
\ca_{j0}$ ($j\in\{1, \ldots m\}$).  Another special case is $q(x)=x^m$, i.e.,
$f_q=\sigma_m$, which is studied in \cite[Prop.~A2]{AMT}, where the solutions
are given as $\Afn_j(x)=P(x) + \ca_{j1}x^{m+1} + \ca_{j0}x^{m+2}$ for a
$j$-independent polynomial $P$ of degree $\leq m+2$ and arbitrary real
constants $\ca_{j1}, \ca_{j0}$. However, we notice that in this case the
corresponding extremal K\"ahler metric $g_q$ given by \eqref{orthotoric-twist}
is orthotoric with respect to the variables $\tilde x_j= 1/x_j$ and functions
$\tilde \Afn_j(\tilde x_j) = \tilde x_j^{m+2}\Afn_j(1/\tilde x_j)$, so the
corresponding extremal K\"ahler metrics are not new. More generally, we may
extend arguments from \cite[Lemma 6]{HFKG1} and \cite[Prop. A2]{AMT} as
follows.

\begin{prop} Let $m\geq 3$. Then the orthotoric K\"ahler
metric~\eqref{orthotoric} is $(f_q,m+2)$-extremal for some positive $f_q$ in
the form \eqref{polarization} if and only if either all $\Afn_j(x)$ are equal
to a $j$-independent polynomial of degree $\leq m+2$ or else the polynomial
$q$ has a root of multiplicity $m$ \textup(possibly at infinity\textup) so
that, up to a simultaneous affine transformation of the $x_j$ in
\eqref{orthotoric}, we may assume that either $q(x)=1$ or $q(x)=x^m$. Then,
$\Afn_j(x)$ are the solutions described in \cite[Prop.~17]{HFKG1} and
\cite[Prop.~A2]{AMT}, respectively. In particular, each extremal K\"ahler
metric of the form \eqref{orthotoric-twist} is either Bochner-flat or
orthotoric with respect to suitable variables.
\end{prop}
\begin{proof} Multiplying \eqref{f-ext-orthotoric} by
$\Delta := \prod_{j<k} (x_j - x_k)$, we get the relation
\begin{equation*}
f_q(\bx)^{m+3} \sum_{j=1}^m \pm  \Delta(\hat x_j)
\frac{\partial^2}{\partial x_j^2}
\Bigl(\frac{\Afn_j(x_j)}{f_q(\bx)^{m+1}}\Bigr)
= \Delta\,\Bigl(\sum_{k=0}^mc_k\sigma_k(\bx)\Bigr),
\end{equation*}
where $\Delta(\hat x_j) = \prod_{i<k \neq j}(x_i - x_k)$ and the signs
$\pm$ are left unspecified. The right hand side in the above equality is a
polynomial of degree $\leq m$ in each variable $x_j$, $\Delta(\hat x_j)$ is
a polynomial of degree $(m-2)$ in any $x_i, i\neq j$ and of degree $0$ in
$x_j$, whereas $f_q(\bx)$ is a polynomial of degree $\leq 1$
in each $x_j$. It follows that for $m\geq 2$,
\begin{equation*}
0=\frac{\partial^{m+1}}{\partial x_j^{m+1}}
\biggl( f_q(\bx)^{m+3} \frac{\partial^2}{\partial x_j^2}
\Bigl(\frac{\Afn_j(x_j)}{f_q(\bx)^{m+1}} \Bigr)\biggr)
= f_q(\bx)^2 \Afn_j^{(m+3)}(x_j),
\end{equation*}
showing that each $\Afn_j$ must be a polynomial of degree $\leq m+2$.

Now let $k \neq j$ be fixed indices. Multiplying \eqref{f-ext-orthotoric} by
$x_j-x_k$ and letting $x_j=x=x_k$ leads to the vanishing of
\begin{equation}\label{key-relation}
(f_0 + x f_1)^2 P_{jk}''(x)
+ 2(f_0 + xf_1)(f_1 + xf_2) x^{m+2}\Bigl(\frac{P'_{jk}(x)}{x^{m+1}}\Bigr)'
+ (f_1+xf_2)^2x^{m+3}\Bigl(\frac{P_{jk}(x)}{x^{m+1}}\Bigr)'',
\end{equation} 
where $P_{jk}(x)= \Afn_j(x)-\Afn_k(x)$. Here each $f_k = \sum_{r=0}^m q_{r+k} \hat
\sigma_r$, with $\hat \sigma_r$ denoting the $r$-th elementary symmetric
function of the variables $x_i: i\neq j, k$ (and letting $\hat \sigma_r=0$ for
$r>m-2$), is a polynomial of degree $\leq 1$ in each $x_i:i\neq j, k$.
Equivalently, $f_k: k\in\{0, 1, 2\}$ can be viewed as affine functions in the
variables $\hat \sigma_1, \ldots \hat \sigma_{m-2}$ and thus
\eqref{key-relation} can be viewed as a polynomial of degree $\leq 2$ in $\hat
\sigma_1, \ldots \hat \sigma_{m-2}$.  By making an simultaneous affine change
of the variables $x_j$ in \eqref{orthotoric} if necessary (which preserves the
orthotoric structure of the metric, see \cite{HFKG1}), we can assume without
loss that $q_0\neq 0$, i.e., $f_0 \neq 0$. We thus consider the following
three cases.

\smallbreak\noindent {\bf Case 1.} $f_0, f_1, f_2$ are linearly independent
affine functions of $\hat \sigma_1, \ldots \hat \sigma_{m-2}$. Then, using
$f_0, f_1, f_2$ as independent variables, and considering the coefficients of
$f_2^2$, $f_0f_1$ and $f_0^2$ in \eqref{key-relation} yields that $P_{jk}(x)$
must belong to the common kernel of the ODEs $P''(x)=0$, $(P'(x)/x^{m+1})'=0$
and $(P(x)/x^{m+1})''=0$. The latter is trivial, thus showing that $P_{jk}(x)
\equiv 0$ in this case, i.e., $\Afn_j(x)=P(x)$ must be a $j$-independent
function.

\smallbreak\noindent {\bf Case 2.} $f_0, f_1, f_2$ span a $2$-dimensional
subspace of affine functions of $\hat \sigma_1, \ldots \hat \sigma_{m-2}$.  In
this case, \eqref{key-relation} is a polynomial of degree $2$ in two
independent variables in the span of $f_0, f_1, f_2$, which places three
relations involving $P''(x)$, $(P'(x)/x^{m+1})'$ and $(P(x)/x^{m+1})''$. Using
their functional independence, we conclude again that $P''(x)=0$,
$(P'(x)/x^{m+1})'=0$ and $(P(x)/x^{m+1})''=0$, i.e.,
$P_{jk}(x)=\Afn_j(x)-\Afn_k(x)=0$ so that $\Afn_j(x)=P(x)$ is a $j$-independent
function.

\smallbreak\noindent {\bf Case 3.}  $f_0, f_1, f_2$ span a $1$-dimensional
subspace of affine functions in $\hat \sigma_1, \ldots \hat \sigma_{m-2}$.
Thus, in this case, $f_1 = \lambda_1 f_0$ and $f_2 = \lambda_2 f_0$ for some
$\lambda_1, \lambda_2\in \R$.  The first identity means $q_{r+1} = \lambda_1
q_r$ for $r\in\{0, \ldots m-1\}$ whereas the second identity is equivalent to
$q_{r+2} = \lambda_2 q_r$ for $r\in\{0, \ldots m-2\}$. As we have assumed $q_0
\neq 0$, we conclude that $\lambda_2=\lambda_1^2$ and then $q_r = \lambda_1^r
q_0$, i.e., $q(x)=q_0(1+ \lambda_1 x)^m$. It thus follows that either $q(x)=1$
(i.e., $\lambda_1 =0$) and then \eqref{f-ext-orthotoric} describes the
extremal K\"ahler condition of an orthotoric metric, which has been analysed
in \cite[Prop. 15]{HFKG1}. Otherwise, by making a simultaneous affine change
of $x_j$ in \eqref{orthotoric}, we can assume $q(x)=x^m$ and
\eqref{f-ext-orthotoric} then reduces to finding $(\sigma_m, m+2)$-extremal
metrics, which has been accomplished in \cite[Prop.~A2]{AMT}.

\medbreak\noindent To summarize, we have proven that one of the following
holds:
\begin{bulletlist}
\item $q=1$ and $\Afn_j(x)= P(x) + \ca_{j1}x + \ca_{j0}$, for a $j$-independent
  polynomial $P$ of degree $\leq m+2$;
\item up to a simultaneous affine transformation of the $x_j$ in
  \eqref{orthotoric}, $q(x)=x^m$ and $\Afn_j(x) = P(x) + \ca_{j1}x^{m+1} + \ca_{j0}
  x^{m+2}$ for a ($j$-independent) polynomial $P$ of degree $\leq m+2$;
\item $\Afn_j(x)=P(x)$ are all equal to a polynomial $P$ of degree $\leq m+2$.
\end{bulletlist}
In the third case, the orthotoric K\"ahler metric \eqref{orthotoric} is
Bochner-flat (see e.g. \cite[Prop.~17]{HFKG1} or \cite{Bryant}) and any $q$
provides a solution to \eqref{f-ext-orthotoric} (see
Proposition~\ref{BF-local}).  By result of Webster~\cite{Webster}, any CR
$q$-twist of $g$ is again a Bochner-flat K\"ahler metric, which completes the
proof.
\end{proof}

\begin{rem} Similar arguments yield a classification of $(f_q,\wt)$-extremal
orthotoric metrics, where $\wt\in\R\setminus\{1, \ldots m+2\}$ and $m \geq 3$.
Indeed, as shown in \cite{AMT}, in this case we have to consider the
equation
\begin{equation}\label{f-ext-orthotoric-2}
-\sum_{j=1}^m \frac{f_q(\bx)^{\wt+1}}{\Delta_j}\frac{\partial^2}{\partial
  x_j^2}\biggl(\frac{\Afn_j(x_j)}{f_q(\bx)^{\wt-1}}\biggr)= \Scal_{f_q,\wt}(g) =
\sum_{k=0}^m c_k\mm_k.
\end{equation}
Multiplying by $x_j-x_k$ and letting $x_j=x=x_k$ leads again to the conclusion
that one of the following three cases occurs: (1) $q= 1$ and $\Afn_j(x)=P(x) +
\ca_{j1}x + \ca_{j0}$ for a polynomial $P$ of degree $\leq m$ by the
classification in \cite[Prop. 15]{HFKG1}, or, (2) up to a simultaneous affine
change of the variables $x_j$ in \eqref{orthotoric}, $q= x^m$ and $\Afn_j(x)=
P(x) + \ca_{j1} x^{\wt-1} + \ca_{j0} x^\wt$ according to \cite[Prop. A2]{AMT},
or (3) $\Afn_j(x)= P(x)$ are $j$-independent. In the third case, multiplying
\eqref{f-ext-orthotoric-2} with $\Delta$ leads to the equation
\begin{equation*}
0=\frac{\partial^{m+1}}{\partial x_j^{m+1}} \Bigl( f_q(\bx)^{\wt+1}
\frac{\partial^2}{\partial x_j^2}\Bigl(\frac{P(x_j)}{f_q(\bx)^{\wt-1}}\Bigr)\Bigr)
= f_q(\bx)^{\wt-m} \frac{\partial^2}{\partial x_j^2}
\Bigl(\frac{P^{(m+1)}(x_j)}{f_q(\bx)^{\wt-m-2}}\Bigr).
\end{equation*}
Letting $x_j=x$, $f_k = \sum_{r=0}^m q_r \hat \sigma_{r-k}$ where $\hat
\sigma_{r}$ denotes the $r$-th elementary symmetric function of $x_i, i\neq
j$ with $\hat \sigma_r=0$ for $r\geq m$, the above conditions reduce to
the vanishing of
\begin{equation*}
(f_0 + xf_1)^2 P^{(m+3)}(x)  - 2(\wt-m-2) f_1(f_0+ xf_1) P^{(m+2)}(x)
+ (\wt-m-2)(\wt-m-1)f_1^2P^{(m+1)}(x).
\end{equation*}
If $f_0$ and $f_1$ are linearly independent affine functions of $\hat
\sigma_1, \ldots \hat \sigma_{m-1}$ (and as $\wt\neq m+1,m+2$ by assumption),
this implies $P^{(m+1)}(x)=0$, i.e., $P$ must be a polynomial of degree $\leq
m$ and the metric \eqref{orthotoric} is flat (see \cite[Prop.~17]{HFKG1}).
These are precisely the solutions described in
\cite[Prop.~A1]{AMT}. Otherwise, either $f_0=0$ (i.e., $q=x^m$) or $f_1 =
\lambda f_0$ (i.e., $q(x)=q_0(1+ \lambda x)^m$), so we are again in a
situation covered by \cite[Prop.~A2]{AMT} and \cite[Prop.~15]{HFKG1}.
\end{rem}

\section{The Calabi problem and non-existence results}\label{s:global}

\subsection{The Calabi problem for $(f,\wt)$-extremal K\"ahler metrics}

Let $(M, J)$ be a compact connected complex manifold of real dimension $2m$,
and $\Omega \in H^2(M, \R)$ a K\"ahler class.  As observed in \cite{AM, FO,
  lahdili1}, many features of the theory of extremal K\"ahler metrics extend
naturally to the $(f,\wt)$-extremal case. In particular, one can formulate a
weighted version of the Calabi problem~\cite{calabi} which seeks an
$(f,\wt)$-extremal K\"ahler metric $(g,\omega)$ with $\omega\in\Omega$.  We
pin down the function $f$ indirectly by fixing first a quasi-periodic
holomorphic vector field with zeroes $K$ generating a torus $\T \leq \Aut^r(M,
J)$ inside the reduced group of automorphisms of $(M, J)$ (see
e.g.~\cite{gauduchon-book}), and secondly, a real constant $\av>0$ such that for
any $\T$-invariant K\"ahler metric $(g, \omega)$ with $\omega \in \Omega$, the
Killing potential $f$ of $K$ with respect to $g$, normalized by $\int_M f
\omega^m/m! = \av$, is positive on $M$, see \cite[Lemma 1]{AM}.

\begin{problem}\label{Calabi-Kahler-problem}  Is there a $\T$-invariant
K\"ahler metric $(\tilde g, \tilde \omega)$ on $(M, J)$ with $\tilde{\omega}
\in \Omega$, which is $(\tilde f, \wt)$-extremal where $\tilde f$ is the
K\"ahler potential of $K$ with respect to $\tilde g$ determined by $\int_M
\tilde f \tilde \omega^m/m! = \av$? We refer to such metrics as $(K, \av,
m+2)$-extremal.
\end{problem}

\begin{rem}\label{r:potential} It is easy to check that if $(g, \omega)$
and $(\tilde g, \tilde \omega)$ are two $\T$-invariant K\"ahler metrics in
$\Omega$ with K\"ahler forms $\tilde \omega = \omega +\d\d^c \varphi$ for a
$\T$-invariant smooth function $\varphi$ on $M$, then the corresponding
$\av$-normalized Killing potentials $\tilde f$ and $f$ of $K$ are related by
\begin{equation}\label{potential}
\tilde f = f  + \d^c\varphi(K).
\end{equation}
Indeed, the proof of \cite[Lemma 1]{AM} shows that $\tilde f = f \circ \Phi$
for some diffeomorphism $\Phi$ of $M$. It thus follows that if $f$ is the
$\av$-normalized Killing potential of $K$ with respect to $g$, then $\tilde f$
is the unique Killing potential of $K$ with respect to $\tilde g$ such that
$\tilde f(M)= f(M)$. The latter property holds for $\tilde f$ defined by
\eqref{potential}, as can be seen by considering points of minima and maxima
for $f$ and $\tilde f$ (at which $K=J \grad_g f = J\grad_{\tilde g} \tilde f$
vanishes).
\end{rem}


\subsection{The Calabi problem for $(\cae,\wt)$-extremal Sasaki metrics}

Let $\Sm$ be a compact connected $(2m+1)$-manifold.
Following~\cite{BG-book,BGS}, one can extend the Calabi problem to an
analogous problem in Sasaki geometry by fixing a nowhere zero vector field $X$
as the candidate for the Sasaki--Reeb vector field of the extremal Sasaki
structure on $\Sm$, together with a complex structure $J_X$ on the quotient
bundle $\Ds_X$ of $T\Sm$ by the span of $X$. If $(\Ds, J, \sas)$ is a Sasaki
structure with $X_\sas= X$ then $\Ds$ is transverse to $X$ and so the
projection onto $\Ds_X$ is a bundle isomorphism, and we can require in
addition that this isomorphism intertwines $J$ and $J_X$. The corresponding
Sasaki structures on $\Sm$ are completely determined by their contact
distributions, or equivalently, their contact forms $\cf$, with $\cf(X)=1$ and
$\d\cf(X,\cdot)=0$.

\begin{defn}\label{d:sasaki-polarization}~\cite{BGS}
The subspace $\cS(X,J_X)$ of $\Omega^1(\Sm)$ whose elements $\cf$ are contact
forms of Sasaki structures compatible with $(X,J_X)$ is called a \emph{Sasaki
  polarization} of $(\Sm, \Ds, J, X)$. We also fix a torus $\T$ in the
automorphism group $\Aut(\Sm,X,J_X)$ and let $\cS(X, J_X)^\T$ denote the
$\T$-invariant elements of $\cS(X,J_X)$.
\end{defn}

One can now imitate constructions in K\"ahler geometry by using the basic de
Rham complex $\Omega^\bullet_X(\Sm)=\{\alpha\in \Omega^\bullet(\Sm):
\imath_X\alpha = 0 =\cL_X \alpha\}$, with differential $\d_X$ given by
restriction of $\d$ (which evidently preserves basic forms). Since
$\{\alpha\in \Wedge^kT^*\Sm:\imath_X\alpha=0\}$ is naturally isomorphic to
$\Wedge^k\Ds_X^*$, $J_X^*$ is also well defined on $\Omega^\bullet_X(\Sm)$,
yielding a twisted differential $\d^c_X$. The same holds for the subspace
$\Omega^\bullet_X(\Sm)^\T$ of $\T$-invariant basic forms.
Following~\cite[Lemma~3.1]{BGS} and \cite[Prop.~7.5.7]{BG-book}, $\cS(X,
J_X)^\T$ is an open subset of an affine space modelled on
$C^\infty_{\Sm,0}(\R)^\T \times \Omega^1_{X,\mathrm{cl}}(\Sm)^\T$, where
$C^\infty_{\Sm,0}(\R)^\T$ denotes the quotient by constants of the space of
smooth $\T$-invariant functions on $\Sm$, and
$\Omega^1_{X,\mathrm{cl}}(\Sm)^\T$ denotes the basic $\T$-invariant closed
$1$-forms on $\Sm$.  Indeed, for any two elements $\cf,\tilde\cf\in \cS(X,
J_X)^\T$, $\tilde\cf-\cf$ is basic and so
\begin{equation}\label{parametrization}
\tilde \cf = \cf + \d_X^c \varphi + \alpha
\end{equation}
for a $\T$-invariant smooth function $\varphi$, and a basic $\T$-invariant
closed $1$-form $\alpha$.  It follows from \eqref{parametrization} that the
induced K\"ahler forms on local quotients $(M, J)$ of $\Sm$ by $X$ are linked
by $\tilde \omega = \omega + \d\d^c \varphi$, i.e., belong to the same
K\"ahler class. Moreover any $K$ in the Lie algebra of $\T$ is CR for both CR
structures $(\Ds, J)$ and $(\tilde \Ds, \tilde J)$ induced by $\cf$ and
$\tilde\cf$, and hence induces on any such $M$ a Killing vector field, also
denoted $K$, for both $\omega$ and $\tilde\omega$, with respective Killing
potentials pulling back to $f = \cf (K)$ and $\tilde f = \tilde \cf(K)= f +
\d^c \varphi (K) + \alpha(K)$. Notice that by the $\T$-invariance and
closedness of $\alpha$, the term $\alpha(K)$ is a constant.

\begin{lemma}\label{l:Reeb} Let $(\Sm, \Ds, J)$ be a compact CR manifold
and $\sas, \cae \in \crJ_+(\Ds, J)$ with $[\cae,\sas]=0$. Let $X=X_\sas$,
$K=X_\cae$. Then for any $\tilde \cf\in \cS(X, J_X)$ with $\cL_K\tilde\cf=0$,
$K$ is a Sasaki--Reeb vector field for the induced CR structure
$(\tilde\Ds,\tilde J)$.
\end{lemma}
\begin{proof} As $K$ is a CR vector field by construction, we need to check
that $\tilde f:=\tilde\cf(K) >0$.  We let $\cf\in \cS(X, J_X)$ be the contact
form of $(\Ds,\sas)$. Since $K$ is contact with respect to
$\tilde\Ds:=\ker\tilde\cf$, we have $K=\tilde f X-
(\d\tilde\cf\restr{\tilde\Ds})^{-1}(\d \tilde f\restr{\tilde\Ds})$ and hence
\[
\cf(K) = \tilde f - (\d\tilde\cf\restr{\tilde\Ds})^{-1}(\d \tilde
f\restr{\tilde\Ds},\cf\restr{\tilde\Ds})
\]
Evaluating this relation at a global minimum $\pt$ of $\tilde f$ we obtain
$\tilde f(\pt) = \cf(K) (\pt) >0$.
\end{proof}

We now specialize the above set-up. First, we fix a nowhere zero vector $K$
and let $\T$ be the torus in $\Aut(\Sm,X,J_X)$ generated by $X$ and $K$.  In
addition, following~\cite{FOW}, we fix a basepoint $\cf\in \cS(X,J_X)^\T$ with
corresponding Sasaki structure $(\Ds,J,\sas)$, and restrict attention to the
affine slice of $\cS(X, J_X)^\T$ consisting of contact forms $\tilde\cf$
related to $\cf$ by~\eqref{parametrization} with $\alpha=0$. We may thus
identify this slice with
\begin{equation}\label{eq:param}
\Sc(X,J_X)^\T:=\{\varphi\in C^\infty_{\Sm,0}(\R)^\T\,|\, \cf_\varphi:=\cf +
\d^c_X\varphi \text{ is a contact form}\}
\end{equation}
We also write $(\Ds_\varphi,J_\varphi,\sas_\varphi)$ for the Sasaki structure
induced by $\cf_\varphi$ for $\varphi \in \Sc_K(X,J_X)^\T$, and let
$\cae_\varphi=\cf_{\Ds_\varphi}(K)$ and $\cae=\cf_\Ds(K)$. In view of
Lemmas~\ref{l:potential} and \ref{l:Reeb}, we now have an analogue of
Problem~\ref{Calabi-Kahler-problem} for $(\cae,\wt)$-extremal Sasaki metrics.

\begin{problem}\label{Calabi-Sasaki-problem} Given a compact CR
manifold $(\Sm, \Ds, J)$ of Sasaki type and $\sas, \cae\in \crJ_+(\Sm, \Ds,
J)$ with $[\sas, \cae]=0$, is there $\varphi \in \Sc(X, J_X)^\T$ such that
$(\Ds_\varphi, J_\varphi, \sas_\varphi)$ is $(\cae_\varphi,\wt)$-extremal?
\end{problem}
In the case $\sas=\cae$, Problem~\ref{Calabi-Sasaki-problem} reduces to the
search for extremal Sasaki metrics in a given Sasaki polarization, see
\cite{BGS}, which has been studied in many places, see
e.g.~\cite{BGS,BV,CSz,legendre2,MSY1,MSY2,V}. We notice that
Problems~\ref{Calabi-Kahler-problem} and \ref{Calabi-Sasaki-problem} are
naturally linked in the regular case, via Examples~\ref{e:regular},
\ref{e:regular-(xi,wt)-extremal} and Remark~\ref{r:potential}. Indeed, the
parametrization~\eqref{eq:param} implies that for any $\varphi\in \Sc(X,
J_X)^\T$, the K\"ahler potentials $\tilde f= \cf_{\varphi}(K)$ and $f=\cf(K)$
of $K$ are linked on a Sasaki--Reeb quotient $(M, J)$ via $\tilde f= f + \d^c
\varphi (K)$, which is consistent with \eqref{potential}.

\begin{rems}\label{r:Calabi-Moser} \begin{numlist}
\item By the equivariant Gray--Moser theorem and Lemma~\ref{l:Reeb} above, for
  each $\varphi \in \Sc(X, J_X)^\T$ the corresponding contact form
  $\cf_\varphi \in \cS(X, J_X)^\T$ is equivalent to $\cf$ by an identity
  component $\T$-equivariant diffeomorphism $\Phi$ of $\Sm$.  Pulling back
  $J_\varphi$ by $\Phi$ gives a CR structure $J_{\varphi,\cf}$ in the space
  ${\mathcal C}_+(\Sm, \Ds)^\T$ of $\T$-invariant $\Ds$-compatible CR
  structures on $(\Sm, \Ds)$ introduced in Section~\ref{s:GIT}, where $\T\leq
  \Con(\Sm,\Ds)$ is the torus generated by the CR vector fields $X_\sas$ and
  $X_\cae$. Thus, Problem~\ref{Calabi-Sasaki-problem} can be viewed as a
  special case of the problem of finding critical points in
  $\mathcal{C}_+(\Sm, \Ds)^\T$ for the square norm of the momentum map of the
  $\Con(\Sm, \Ds)^\T$-action on $\mathcal{AC}_+(\Sm, \Ds)^\T$ defined in Section~\ref{s:GIT}.

\item In view of Theorem~\ref{thm:main}, one may ask
whether, given $(\Sm, \Ds, J, \sas, \cae)$ as in
Problem~\ref{Calabi-Sasaki-problem}, there exists a $(\cae_\varphi,
m+2)$-extremal solution $\varphi\in \Sc(X,J_X)^\T$ iff there exists an
extremal Sasaki structure with contact form $\cf_\varphi\in\cS(K,
J_K)^\T$. This would be useful for studying irregular extremal Sasaki
structures by taking $X$ quasi-regular and $K$ irregular.  However, it is not
clear how to relate the spaces $\Sc(X,J_X)^\T$ and $\cS(K, J_K)^\T$, since
$J_\varphi$ is the lift of $J_K$ to $\Ds_\varphi$ and its projection onto
$T\Sm/\spn{X}$ does not agree with $J_X$ in general, and so the pullback
$J_{\varphi,\cf}=\Phi^* J_\varphi$ with $\Phi^*\cf_\varphi=\cf$ need not
descend to a complex structure in the same Teichm\"uller class as $J$ on the
quotient of $\Sm$ by $X$.
\end{numlist}
\end{rems}


\subsection{Extremal Sasaki structures from ruled complex surfaces}\label{s:ruled-surface}

We now specialize to geometrically ruled complex surfaces and the regular
Sasaki manifolds they define. Let $(M, J)=\pi\colon \Proj(\cO \oplus
\cL)\to\Bm$ be the underlying complex manifold of a projective $\C P^1$-bundle
over a compact Riemann surface $\Bm$, where $\cL$ is a holomorphic line bundle
over $\Bm$ of positive degree $\ell$. Let $K$ be the generator of the
holomorphic $\Sph^1$-action on $(M,J)$ induced by scalar multiplication in
$\cO$. We denote by $(g_\Bm, \omega_\Bm)$ the K\"ahler metric on $\Bm$ of
constant scalar curvature $4(1-\genus)$, where $\genus$ denotes the genus of
$\Bm$. It is well-known (see e.g.~\cite{HFKG3}) that the K\"ahler cone of $(M,
J)$ can be parametrized up to homothety by the cohomology classes of K\"ahler
metrics $(g,\omega)$ given by the Calabi ansatz~\eqref{ruled-kahler} as
described in Section~\ref{s:Calabi}, with $\ah_1=\ell$. For convenience, in
this section, we let $\ah$ denote the real constant $\ah_0/\ell$ and write
$\mc$ for $\mc_1$ and $\mm$ for $\mm_1$, so that
$\ah_0+\ah_1\mm_1=\ell(\mm+\ah)$.

For each $\bh\in \R$ with $|\bh|>1$, $f_\bh := \mm+\bh$ is a positive Killing
potential for $K$ on $(M, g, \omega)$.  The existence of a $(f_\bh,
4)$-extremal K\"ahler metric of the form \eqref{ruled-kahler} on $M$ (up to
homothety) is studied in \cite{AMT,KTF,LeB2}. It is shown there (see
e.g.~\cite[Thm.~1]{AMT}) that such a metric must be obtained from a smooth
function $\Afn(\mc)= P_{\ah,\bh}(\mc)/(\mc+\ah),$ where $P_{\ah,\bh}$ is a
polynomial of degree $\leq 4$ uniquely determined from \eqref{boundary} in
terms of $\ah, \bh$ and $\ell$; thus$(M,g,J,\omega)$ is $(f_\bh, 4)$-extremal
iff $P_{\ah,\bh}(\mc)$ satisfies the positivity condition
\eqref{positive}. Conversely, we have the following result.

\begin{prop}\label{c:ruled} Let $M= P(\cO \oplus \cL) \to \C P^1$ be a ruled
surface and $\Omega= \lambda[\omega]$ a K\"ahler class on $M$ for
$\lambda>0$, $\ah>1$. Let $|\bh|>1$, $f_b=\mm+b$ and $\av= \frac{\lambda^3}{2}
\int_M f_\bh\, \omega^2$.  If the polynomial $P_{\ah,\bh}(\mc)$ is not positive
on $(-1,1)$, $\Omega$ contains no $(K, \av, 4)$-extremal K\"ahler metrics.
\end{prop}
\begin{proof} The proof follows from a slight modification of the arguments
in~\cite[Cor.~1]{lahdili}, taking into account the recent result
\cite[Cor.~1]{lahdili2}. Indeed, suppose for contradiction that $P_{\ah,
  \bh}(\pt_0)=0$ for some $\pt_0\in (-1,1)$, and that $[\omega]$ admits a $(K,
\av, 4)$-extremal K\"ahler metric.

Consider first the case that $P_{\ah, \bh}(\mc)$ is negative somewhere on
$(-1,1)$. If the K\"ahler class $[\omega/2\pi]$ is rational (which is
equivalent to $\ah\in \Q$), we derive a contradiction by
\cite[Cor.~1]{lahdili2} (which implies that the relative weighted Mabuchi
functional must be bounded from below) and \cite[Prop.~2.7]{AMT} (which
concludes otherwise).  If the K\"ahler class $[\omega/2\pi]$ is not rational,
we can approximate it with rational classes $[\tilde \omega/2\pi]$ of the form
\eqref{kahler-class} by taking rational values $\tilde \ah$ close to $\ah$,
and still ensure that $P_{\tilde \ah, \bh}(\mc)$ is negative somewhere on
$(-1,1)$. Furthermore, by the openness of weighted extremal classes
established in \cite[Thm.~2]{lahdili1}, we can assume that $[\tilde \omega]$
admits a $(K, \tilde \av, 4)$-extremal K\"ahler metric with $\tilde\av=
\frac{1}{2} \int_M f_\bh\, \tilde \omega^2$. We get a contradiction as before.

It thus remains to consider the case when $\pt_0\in (-1,1)$ is a double root of
the quadratic $P_{\ah, \bh}(\mc)/(1-\mc)^2$.  In this case, we prove that there
exists a sequence $\tilde \bh_i$ converging to $\bh$, such that
$P_{\ah,\tilde\bh_i}(\pt_0)<0$; by the openness result in
\cite[Thm.~2]{lahdili1}, we can then find $\tilde\bh$ with $P_{\ah, \tilde
  \bh}(\pt_0)<0$ and such that the K\"ahler class $[\omega]$ admits a $(K,
\tilde \av, 4)$-extremal K\"ahler metric with $\tilde\av= \frac{1}{2} \int_M
f_{\tilde \bh} \,\omega^2$, a situation we have already ruled out.

In order to find a sequence as above, it is enough to show that
$\frac{\partial P_{\ah,\bh}}{\partial\bh}(\pt_0) \neq 0$ at each double root
$\pt_0 \in (-1,1)$. The remainder of the proof establishes this technical
fact.

If $\ah = \bh$, we get the natural solution of Corollary~\ref{cor:natural}, in
which case $P_{\ah,\bh}(\mc)>0$ on $(-1,1)$. We can thus assume that $\ah\neq
\bh$, and then, adapting~\cite[(11) \& (31)]{KTF} to our notation, the
polynomial $P_{\ah,\bh}$ is given by
\begin{equation}\label{P}
\frac{2 P_{\ah,\bh}(\mc)}{1-\mc^2}= 2(\mc+\ah) +(1-\mc^2)
\frac{ 3\bht+\ah+s}
{3\bht^2-1},
\end{equation}
where $s=2(1- \genus)/\ell$ and $\bht=(\ah\bh-1)/(\ah-\bh)$.

Suppose that $\frac{\partial P_{\ah,\bh}} {\partial\bh}(\pt_0)=0$ for
$\pt_0\neq\pm1$. We compute that this is equivalent to
\begin{equation*}
2\bht\, s+ 3\bht^2+ 2 \ah\,\bht+1=0.
\end{equation*}
Since $\bht\neq 0$, we may solve for $s$ and substitute back in \eqref{P} to
obtain
\begin{equation*}
\frac{2 P_{\ah,\bh}(\mc)}{1-\mc^2} =
\frac{1+4 \ah\,\bht +4\bht\,\mc-\mc^2}{2\bht}.
\end{equation*}
Now if $P_{\ah,\bh}$ has a double root at $\pt_0\neq\pm 1$, we must have
$\pt_0=2\bht=2(\ah\bh-1)/(\ah-\bh)$, the critical point of this
expression. However, $(\ah\bh-1)^2-(\ah-\bh)^2 = (\ah^2-1)(\bh^2-1) > 0$ for
$|\bh|>1$ and $\ah>1$, so $|\pt_0|>2$ and hence $\pt_0\notin(-1,1)$.
\end{proof}

It is shown in \cite[Prop.~2.12]{AMT} that for a rational K\"ahler class of
the form \eqref{kahler-class}, and for $\mc\in (-1,1)\cap\Q$, $P_{\ah,\bh}(\mc)$
computes a weighted notion of the relative Donaldson--Futaki invariant
associated to the polarized variety $(M, L_{\p,\q})$, where $L_{\p,\q}$ is a
polarization on $M$ corresponding to $\ah = \ah_0/\ell= (2\q/\p\ell) -1$ as
explained in Section~\ref{s:Calabi}. This motivates the following definition.

\begin{defn}\label{weighted-stability}\cite{AMT} Let $(M,L_{\p,\q})$ be a
polarized ruled surface as above, and $\hat Z_{\bh}$ the quasi-periodic (real)
holomorphic vector field on $L_{\p,\q}$, given by the lift of $K$ with respect
to the potential $f_\bh$.  We say that $(M,L_{\p,\q}, \hat Z_\bh)$ is
\emph{analytically relatively $(\hat Z_\bh, 4)$ \textup K-stable with respect
  to admissible test configurations} if $P_{\ah,\bh}(\mc)>0$ on $(-1,1)$.
\end{defn}

Thus, Proposition~\ref{c:ruled} implies that that $(M, J)$ admits a $(K, \av,
4)$-extremal K\"ahler metric in $2\pi c_1(L_{\p,\q})$ iff $(M,L_{\p,\q}, \hat
Z_\bh)$ is analytically relatively $(\hat Z_\bh, 4)$ K-stable with respect to
admissible test configurations.

\subsection{Proof of Theorem~\ref{thm:ruled}} Let $(\Sm_{\p,\q}, \Ds, \sas)$ be
a compact regular contact manifold over a ruled surface $M$ constructed in
Section~\ref{s:Calabi}, and $\cae_\bh \in \con(\Sm_{\p,\q}, \Ds)$ the contact
lift of the generator $K$ of the $\Sph^1$-action on $M$ via potential $\mm +
\bh$. We denote by $\T \leq \con(\Sm_{\p,\q}, \Ds)$ the torus generated by
$\sas$ and $\cae_\bh$, and let $P_{\ah,\bh}$ be the polynomial corresponding
to the K\"ahler class \eqref{kahler-class} with $\ah_0= (2\q/\p) -\ell$ and
$\ah_1=\ell$. We need to show that there exists a $(\cae_\bh,4)$-extremal CR
structure $J\in \mathcal{C}_+(\Sm_{\p,\q},\Ds)^\T$ if and only if
$P_{\ah,\bh}(\mc)>0$ on $(-1,1)$.

Existence follows from Theorem~\ref{thm:main} and the fact that when
$P_{\ah,\bh}(\mc)>0$ on $(-1,1)$, $\Afn(\mc)= P_{\ah,\bh}(\mc)/(\mc+\ah)$
in~\eqref{ruled-kahler} defines an $(f_b, 4)$-extremal K\"ahler metric on $M$
(see \cite[Thm.~1]{AMT}).

We now establish the non-existence claim.  Suppose that $P_{\ah,\bh}(\mc)$ has
a zero on $(-1,1)$. Denote by $J$ the CR-structure in
$\mathcal{C}_+(\Sm_{\p,\q}, \Ds)^\T$ induced by a K\"ahler metric $(g,
\omega)$ on $M$ of the form \eqref{ruled-kahler} and suppose for contradiction
that there exists a $\Ds$-compatible CR structure $J' \in
\mathcal{C}_{+}(\Sm_{\p,\q}, \Ds)^\T$ such that $(\Ds, J',\cae_\bh)$ is an
extremal Sasaki structure and $P_{\ah,\bh}(\mc)$ has a zero on $(-1,1)$. As
$J'$ is $X_{\cae_\bh}$-invariant, on $M$ we obtain another $\omega$-compatible
complex structure $J'$ which is invariant under the $\Sph^1$-action generated
by $K$.  As any compact K\"ahler $4$-manifold admitiing a holomorphic vector
field with zeroes must be a rational or a ruled surface~\cite{CHK}, $(M, J')$
must be a ruled surface too, i.e., $(M, J') = P(\cO \oplus \cL') \to \Bm'$
with $\Bm'$ having the same genus as $\Bm$. Intersection properties of the
zero set of $K$ (which equals the zero and infinity sections in either case)
yield that $\deg(\cL) = \deg(\cL')$.  Since $J$ and $J'$ are compatible with
the same symplectic form $\omega$, the corresponding class $[\omega]$ is of
the form \eqref{kahler-class} on either surface, with the same parameter
$\ah=\ah_0/\ell>1$; furthermore, the Killing potential of $K$ induced by
$\cae_\bh$ in both cases is $f_b = \mm+\bh$ for the same $\bh>1$. It thus
follows that the respective polynomials associated to $[\omega]$ on each ruled
surface $(M,J)$ and $(M, J')$ coincide; we denote them by $P_{\ah,\bh}$.

Now, by Theorem~\ref{thm:main}, the K\"ahler class $[\omega]$ on $(M, J')$
admits a $(K, \av, 4)$-extremal K\"ahler metric which, by
Proposition~\ref{c:ruled}, forces $P_{\ah,\bh}(\mc) > 0$ on $(-1,1)$, a
contradiction. \qed

\begin{rems}\begin{numlist}
\item Theorem~\ref{thm:ruled} yields a $(\cae_\bh,4)$-extremal Sasaki metric
on the Sasaki join $(\Sm_{\wv,\p}, \Ds)$ over $\Bm\times \C P^1_\wv$ (with
weights $w_+ = \q, w_- = \q-\p\ell$, where $\ah= (2\q/\p\ell) -1$) if
$P_{\ah,\bh}(\mc)>0$ on $(-1,1)$. This always happens when $\Bm$ has genus $0$
or $1$, or when $\ah$ is sufficiently large (see e.g.~\cite[Thm.~1]{AMT}).
These extremal Sasaki structures are not new (see~\cite{BT0,BT}) but we have
seen in Section~\ref{s:Calabi} that they can equivalently be obtained from a
$(\tilde f_\bh,4)$-extremal product K\"ahler metric $g_\Bm + p g_\wv$ on
$\Bm\times \C P^1_\wv$.  Taking the limit $\bh\to \infty$ also yields the
extremal K\"ahler metrics on the ruled surfaces constructed in~\cite{calabi,
  TF} as a CR twist of a product metric on $\Bm\times \C P^1_\wv$.

\item To the best of our knowledge, the non-existence result obtained via
Theorem~\ref{thm:ruled} is new, at least for irrational values of $\bh$ (in
which case $\cae_\bh$ generates $\T$, and thus is not quasi-regular).  To
construct specific examples, following \cite{KTF}, on any ruled surface $M$
over a curve of genus $\geq 2$ as above, there exists an explicit $\ah_0(M)>1$
such that for any $\ah\in (1, \ah_0(M)]$ the polynomial
$P_{\ah, \bh_\ah}(\mc)$ has a zero on $(-1, 1)$, where $\bh=\bh_\ah>1$ is the
unique solution of $\ah= \frac{1+\bh^2}{2\bh}$ satisfying $|b|>1$.  Taking
coprime positive integers $\p,\q$ with $\ah= (2\q/\p\ell) -1 \leq
\ah_0(M)$, we obtain a contact manifold $(\Sm_{\p,\q}, \Ds)$ which admits no
$(\cae_\bh,4)$-extremal CR structure $J\in \mathcal{C}_+(\Sm_{\p,\q},\Ds)^\T$.
\end{numlist}
\end{rems}

\subsection*{Acknowledgements} 

We would like to thank Paul Gauduchon, Gideon Maschler and Christina
T{\o}nnesen-Friedman for stimulating conversations. We are grateful to Abdellah Lahdili for his useful comments on the manuscript, and to Roland Pu\v cek for discussions which helped to clarify some of the toric constructions in this paper.

\end{document}